\documentclass[11pt,reqno]{amsart}
\usepackage{amsfonts,amssymb,amsmath,amsopn,amsthm,graphicx}
\usepackage{amsxtra, mathrsfs}
\usepackage{colordvi}
\usepackage[usenames,dvipsnames]{color}
\usepackage{amsfonts,amssymb,amsbsy,amsmath,amsthm,dsfont}

\usepackage{amsthm}
\usepackage{amsbsy}
\usepackage{amstext}
\usepackage{accents}
\usepackage{amsfonts}
\usepackage{mathcomp}
\usepackage{mathtools}
\usepackage{dsfont}
\usepackage{tabularx}
\usepackage{latexsym}
\usepackage{amssymb}
\usepackage{esint}
\usepackage{cite}

\usepackage{lipsum}
\usepackage[many]{tcolorbox}
\usetikzlibrary{decorations.pathreplacing}
 \numberwithin{equation}{section}

\newtcolorbox{leftbrace}{ %
	enhanced jigsaw, 
	breakable, 
	frame hidden, 
	overlay={%
		\draw [
		decoration={brace,amplitude=0.5em},
		decorate,
		ultra thick,
		]
		(frame.south west)--(frame.north west);
	},
	parbox=false,
}

\newtheorem{theorem}{Theorem}[section]
\newtheorem{corollary}[theorem]{Corollary}
\newtheorem{cor}[theorem]{Corollary}
\newtheorem{thm}[theorem]{Theorem}
\newtheorem{lem}[theorem]{Lemma}

\newtheorem{lemma}[theorem]{Lemma}

\newtheorem{example}[theorem]{Example}
\newtheorem{assumptions}[theorem]{Assumptions}
\newtheorem{definition}[theorem]{Definition}

\newtheorem{remark}[theorem]{Remark}

\newtheorem{rem}[theorem]{Remark}
\newtheorem{theo}[theorem]{Theorem}
\newtheorem{prop}[theorem]{Proposition}

\newtheorem{assumption}[theorem]{Assumption}
\newtheorem{defin}[theorem]{Definition}

\newcommand{\Hmm}[1]{\leavevmode{\marginpar{\tiny%
			$\hbox to 0mm{\hspace*{-0.5mm}$\leftarrow$\hss}%
			\vcenter{\vrule depth 0.1mm height 0.1mm width \the\marginparwidth}%
			\hbox to
			0mm{\hss$\rightarrow$\hspace*{-0.5mm}}$\\\relax\raggedright #1}}}

\DeclareMathOperator*{\osc}{osc}

\newcommand\red[1]{\textcolor{red}{#1}}

\newcommand{\R}{\ensuremath{\mathbb{R}}}

\newcommand{\C}{\mathrm{C}}

\def\Xint#1{\mathchoice
    {\XXint\displaystyle\textstyle{#1}}%
    {\XXint\textstyle\scriptstyle{#1}}%
    {\XXint\scriptstyle\scriptscriptstyle{#1}}%
    {\XXint\scriptscriptstyle\scriptscriptstyle{#1}}%
    \!\int}
\def\XXint#1#2#3{\setbox0=\hbox{$#1{#2#3}{\int}$}
    \vcenter{\hbox{$#2#3$}}\kern-0.5\wd0}

\def\dashint{\Xint{\raise4pt\hbox to7pt{\hrulefill}}}

\def\XXiint#1#2#3{\setbox0=\hbox{$#1{#2#3}{\iint}$}
    \vcenter{\hbox{$#2#3$}}\kern-0.5\wd0}

\renewcommand{\epsilon}{\varepsilon}

\renewcommand{\rho}{\varrho}

\renewcommand{\epsilon}{\varepsilon}
\renewcommand{\rho}{\varrho}
\renewcommand{\d}{\:\! \mathrm{d}}

\DeclareMathOperator{\loc}{loc}
\DeclareMathOperator{\diam}{diam}

\numberwithin{equation}{section}
\allowdisplaybreaks

\newcommand{\RN}[1]{%
	\textup{\uppercase\expandafter{\romannumeral#1}}%
}

\newcommand{\dx}{\,\mathrm{d}x}
\newcommand{\dy}{\,\mathrm{d}y}

\newcommand{\dt}{\,\mathrm{d}t}

\newcommand{\ds}{\,\mathrm{d}s}

\newcommand{\be}{\begin{equation}}
\newcommand{\ee}{\end{equation}}
\newcommand{\bea}{\begin{eqnarray}}
\newcommand{\eea}{\end{eqnarray}}
\newcommand{\bean}{\begin{eqnarray*}}
	\newcommand{\eean}{\end{eqnarray*}}

\newcommand{\dist}[2]{\mbox{\rm dist}\,(#1,#2)}
\newcommand{\opname}[1]{\mbox{\rm #1}\,}
\newcommand{\supp}{\opname{supp}}

\newlength{\wex}  \newlength{\hex}
\newcommand{\ass}[1]{Let Assumptions~\ref{assump1} hold  in a bounded Lipschitz domain $\Gw$}

       \def\gd{\delta}      
                         
       \def\vgf{\varphi}    \def\gh{\eta}

      \def\gw{\omega}
                \def\gz{\zeta}

\def\Gw{\Omega}              

\begin{document}
\renewcommand{\refname}{References} 
\renewcommand{\abstractname}{Abstract} 
\title[Wolff class Fuchsian potential]{Positive solutions of quasilinear elliptic equations with Fuchsian potentials in\\ Wolff class}
\author[R.~Kr.~Giri]{Ratan Kr. Giri}
\address{Ratan Kr. Giri\\
	\newline
	Department of Mathematics, Technion - Israel Institute of Technology, 
	3200 Haifa, Israel 
	\newline
Department of Mathematics, The LNM Institute of Information Technology, Jaipur - 302031, India}
\email{giri90ratan@gmail.com, ratan.giri@lnmiit.ac.in}
\author[Y.~Pinchover]{Yehuda Pinchover}
\address{Yehuda Pinchover\\
	\newline
	Department of Mathematics, Technion - Israel Institute of Technology\\
	3200 Haifa, Israel}
\email{pincho@technion.ac.il}

\date{\today}
\subjclass[2010]{Primary 35B53; Secondary 35B09, 35J62, 35B40}
\keywords{Fuchsian singularity, Morrey spaces, Kato Class, Wolff class, Liouville theorem, Quasilinear elliptic equation, $(p,A)$-Laplacian}

\begin{abstract}
Using  Harnack's inequality and a scaling argument we study Liouville-type theorems and the asymptotic behaviour of positive solutions near an isolated singular point $\zeta \in \partial\Omega\cup\{\infty\}$ for the quasilinear elliptic equation
$$-\text{div}(|\nabla u|_A^{p-2}A\nabla u)+V|u|^{p-2}u =0\quad\text{ in } \Omega,$$
where $\Omega$ is a domain in $\mathbb{R}^d$, $d\geq 2$, $1<p<d$, and $A=(a_{ij})\in L_{\rm loc}^{\infty}(\Omega; \mathbb{R}^{d\times d})$ is a symmetric and locally uniformly positive definite matrix. It is assumed that the potential $V$ belongs to a certain Wolff class and has  a generalized Fuchsian-type singularity at an isolated point $\zeta\in \partial \Omega \cup \{\infty\}$.
\end{abstract}
\makeatother

\maketitle

\section{Introduction}
This paper studies positive Liouville-type theorems and removable singularity theorems for the $(p,A)$-Laplacian type elliptic equation  
\begin{equation}\label{plaplace}
	Q(u)=Q_{p,A,V}(u) :=
	 -\Delta_{p,A}(u)+V|u|^{p-2}u =0\qquad \text{in } \Omega
\end{equation}
 near an isolated singular point $\zeta\in \partial \Omega\cup\{\infty\}$. Here $\Omega$ is a domain in $\mathbb{R}^d$ ($d\geq 2)$ with boundary $\partial\Omega$, $1<p<d$, $V$ is a real valued potential belonging to a certain Wolff space $\mathfrak{W}^p_{\loc}(\Omega)$ (see Definition~\ref{wolff_cls}), and  
 $$\Delta_{p,A}(u):=\text{div}(|\nabla u|_A^{p-2} A\nabla u)$$ is the $(p,A)$-Laplacian, where $A=(a_{ij})\in L_{\rm loc}^{\infty}(\Gw;\R^{d\times d})$ is a symmetric and locally uniformly positive definite matrix valued function, and
 \be\nonumber
  |\xi|_{A}^2 =|\xi|_{A(x)}^2
 :=
 A(x)\xi\cdot\xi=\!\!\!\sum_{i,j=1}^d a_{ij}(x)\xi_i\xi_j
 \;\; \;\;\;x\in\Gw, \;\xi=(\xi_1,...,\xi_d) \!\in \!\R^d.
 \ee
 
 Liouville theorems and asymptotic behaviours of positive solutions of second-order quasilinear elliptic operators near an {\em isolated singular point} $\zeta\in \partial \Omega\cup\{\infty\}$ was studied extensively in the past few decades, see \cite{frass_pinchover,frass_pinchover2,HKM,LLM,LS,LS_1,S,V} and references therein. In particular, we mention the work of Frass and Pinchover\cite{frass_pinchover}, where Liouville theorems and  removable singularity theorems for positive solutions of \eqref{plaplace} have been obtained   for the entire range $1\!<\!p\!<\!\infty$ under the assumptions that  $A$ is the identity matrix, and $V\!\in\! L^\infty_{\loc}(\Omega)$ has a {\em pointwise Fuchsian-type singularity at}  $\gz\! \in \!\{0,\infty\}$, namely, 
 \begin{equation*}\label{eq_pFs}
 	|V(x)|\leq \frac{C}{|x|^p}\qquad \mbox{near } \gz.
 \end{equation*} 
 Furthermore, in the same paper and  in \cite{frass_pinchover2}, the asymptotic behaviour of the {\it quotient} of two positive solutions near the singular point $\zeta$ has been established. The outcomes in \cite{frass_pinchover,frass_pinchover2} extend  to the quasilinear case results obtained in \cite[and references therein]{pinchover} for positive classical solutions of a second-order {\em linear} elliptic operators in a  non-divergence form.
 
Recently, based on criticality theory for $Q_{p,A,V}$ with $V$ in a local Morrey space $M^q_{\loc}(p;\Omega)$ (see Definition~\ref{morrey_sp}) established in \cite{pinchover_psaradakis}, the aforementioned  results were further extended in \cite{giri_pinchover} to the case where $A=(a_{ij})\in L_{\rm loc}^{\infty}(\Gw;\R^{d\times d})$ is a symmetric and locally uniformly positive definite matrix valued function, and $V$ belongs to   $M^q_{\loc}(p;\Omega)$ which has a generalized Fuchsian-type singularity at $\gz$.  More precisely, let  $\mathcal{A}_R:=\{x\in\mathbb{R}^d: R/2\leq |x|<3R/2\}$. $V\in M^q_{\loc}(p;\Omega)$ is said to have  a {\em generalized  Fuchsian-type singularity at $\gz$} if for every relative punctured neighbourhood $\Omega' \subset \Omega$ of $\zeta$, there exists a positive constant $C$ and $R_0>0$ such that
\begin{align}\label{fuch_morrey}
	\begin{cases}
		\||x|^{p-{d}/{q}}\, V\|_{M^q(p;\mathcal{A}_R\cap \Omega')}\leq C  & \text{if } p\neq d, \\[2mm]
		\|V\|_{M^q(d;\mathcal{A}_R\cap\Omega')}\leq C & \text{if } p=d, 
	\end{cases}
\end{align}
for all $0<R<R_0$ if $\gz=0$, and $R>R_0$ if $\gz=\infty$. 

The results in  \cite{giri_pinchover} covers the full range of $1<p<\infty$ as in \cite{frass_pinchover,frass_pinchover2}. Note that for $p>d$,  $M^q_{\loc}(p;\Omega)= L^1_{\loc}(p;\Omega)$, which is mainly the weakest assumption for the well definiteness of weak solutions, and for ensuring  the local Harnack inequality and the  H\"older continuity of weak solutions of the equation $Q(u)=0$.

On the other and, by Lemma~\ref{lem1} and remarks \ref{remrk} and \ref{remrk1}, the local Morrey space $M^q_{\loc}(p;\Omega)$ for $1<p<d$ is a proper subset of the local Stummel-Kato class $K^p_{\loc}(\Omega)$ (see Definition~\ref{kato_1}), which in turn is a proper subset of the local Wolff class $\mathfrak{W}^p_{\loc}(\Omega)$ (see Definition~\ref{wolff_cls}). It is important to remark here that when $1<p<d$, the assumption that $V\in M^q_{\loc}(p;\Omega)$ is the weakest to ensure the local H\"older continuity of weak solutions of \eqref{plaplace} \cite{pinchover_psaradakis}. That is, if $V\in K^p_{\loc}(\Omega)$ or  $V\in \mathfrak{W}^p_{\loc}(\Omega)$, then weak solutions of \eqref{plaplace} may not be H\"older continuous, see \cite{chiarenza et al}, \cite{Difazio} and \cite{ragusa_zamboni} for counter examples.

 
 The purpose of this paper is to  extend the results obtained in \cite{frass_pinchover, frass_pinchover2, giri_pinchover} for $1<p<d$. We study the asymptotic behaviour of positive weak solutions near a Fuchsian-type singular point $\zeta$, and prove removable singularity theorems for positive solutions of \eqref{plaplace} when the potential $V$ lies in a local Wolff space and having a Fuchsian-type singularity at $\zeta$ (see Definition~\ref{fuchsian}). More precisely, after establishing the needful parts of criticality theory for $Q_{p,A,V}$ with $V \! \in \!\mathfrak{W}^p_{\loc}(\Omega)$, we study the uniqueness (up to a multiplicative constant) of certain positive weak solutions and the asymptotic behaviour of the quotients of two positive weak solutions of the quasilinear elliptic equation \eqref{plaplace} near the singular point $\zeta$. Our main tool is a uniform Harnack inequality and a scaling argument using  the quasi-invariance of  \eqref{plaplace} under the map $x\mapsto Rx$, where $x\in \Gw$ and $R>0$. The novelty of this work lies in the fact that  $V$ is assumed to be in the Wolff class $\mathfrak{W}^p_{\loc}(\Omega)$ which implies that solutions of \eqref{plaplace} are not necessarily H\"older continuous. Another difficulty arises since  $\mathfrak{W}^p(\omega)$ is not a Banach space for $p>2$ for subdomains $\gw\Subset\Gw$. 
 
 
 The paper is organized as follows. In Section \ref{sec_pre}, we  recall various function spaces such as Morrey, Kato and Wolff spaces, and provide basic results related to these spaces. Section~\ref{pde_wolff} is devoted to  criticality theory  for  the operator $Q_{p,A,V}$ with $V$ in the local Wolff class  $\mathfrak{W}^p_{\loc}(\Omega)$. In particular, we prove the corresponding Harnack convergence principle, weak maximum, and weak comparison principles. Finally, in Section~\ref{FCH}, we introduce the notion of a Fuchsian-type singularity, discuss some related notions, and prove the  main results of the paper. In particular, we present a uniform Harnack inequality and a ratio limit theorem for positive solutions defined near the singular point $\zeta$. Finally, we obtain a positive Liouville-type theorem under the further assumption that $Q$ has a {\em weak} Fuchsian-type singularity at $\zeta$.
 
\section{Morrey, Kato and Wolff spaces}\label{sec_pre}
We begin with some standard notation.  Let $\Omega$ be a domain  (i.e., a nonempty open connected set) in $\mathbb{R}^d$, $d\geq 2$. We denote the ball in $\mathbb{R}^d$ of radius $r>0$ centered at $x$ by $B_r(x)$.  Similarly, $S_r(x):= \partial B_r(x)$ denotes the sphere of radius $r>0$ centered at $x$, and we set $B_r:=B_r(0)$, $S_r:=S_r(0)$.  Further, denote by $B_r^*:=\mathbb{R}^d\setminus\overline{B_r}$ and $(\R^d)^*:=\R^d\setminus \{0\}$ the corresponding exterior domains. 

Let $f,g\in C(\Omega)$ be two positive functions. By $f\asymp g$ in $\Omega$, we mean that there exists positive constant $C$ such that 
$$C^{-1}g(x)\leq f(x)\leq C g(x)\quad \text{ for all } \,x\in \Omega.$$
We write $\Omega_1\Subset\Omega_2$ if $\Omega_2$ is open and $\overline{\Omega}_1$ is compact (proper) subset of $\Omega_2$. By a {\em compact exhaustion} of a domain $\Gw$,  we mean a sequence $\{\Gw_i\}_{i=1}^\infty$ of smooth, relatively
compact domains in $\Gw$ such that $\Gw_1 \neq \emptyset$,  $\Gw_i \Subset \Gw_{i + 1}$, and 
$\cup_{i = 1}^{\infty} \Gw_i = \Gw$.  
Finally, throughout the paper, $C$ refers to a positive constant which may vary from  line to line.

\medskip

In this section we introduce various function spaces such as Morrey, Kato and Wolff spaces, and discuss the relationships between them. 
\subsection{Morrey spaces}
We recall certain classes of Morrey spaces.
\begin{defin}[Morrey spaces]{\em 
	Let $f\in L^1_{\loc}(\Omega)$, for $q\in[1, \infty]$, we say that $f$ belongs to the local Morrey space $M^q_{\loc}(\Omega)$ if for any $\omega\Subset\Omega$
	$$\|f\|_{M^q(\omega)}:= \sup_{\substack{y\in \omega\\0<r<\text{diam}(\omega)}}\frac{1}{r^{d/q'}}\int_{\omega\cap B_r(y)}|f(x)|\dx<\infty,$$
	where $q' \!=\! q/(q\!-\!1)$ is the H\"older conjugate exponent of $q$. An application of H\"older inequality implies that $L^q_{\loc}(\Omega)\!\subsetneq\! M^q_{\loc}(\Omega)\!\subsetneq\! L^1_{\loc}(\Omega)$ for any $q\!\in\!(1, \infty)$. 
}
\end{defin}
Next we recall a special local Morrey space $M^q_{\loc}(p;\Omega)$ which depends on the underlying exponent $1<p<\infty$. 
\begin{defin}[Local Morrey spaces]\label{morrey_sp}{\em 
	Let
	$$M^q_{\loc}(p;\Omega):=\begin{cases}
	M^q_{\loc}(\Omega) \mbox{ with } q>d/p &\mbox{ if } p<d,\\
	L^1_{\loc}(\Omega)  & \mbox{ if } p>d,
	\end{cases}$$
	while for $p=d$, a function $f\in L^1_{\loc}(\Gw)$ is said to belong to the Morrey space $M^q_{\loc}(d;\Omega)$ if for some $q>d$ and any $\omega\Subset\Omega$
	$$ \|f\|_{M^q(d;\omega)}:= \sup_{\substack{y\in \omega\\0<r<\text{diam}(\omega)}}\varphi_q(r)\int_{\omega\cap B_r(y)}|f|\dx<\infty,$$
	where $\varphi_q(r)\!:=\!\log ^{q/d'}\!\!\left(\!\!\frac{\text{diam}(\omega)}{r}\!\!\right)\!$ (see \cite[Theorem~1.94]{maly_ziemer}, and references therein).
}
\end{defin}

\subsection{ Stummel-Kato class}
Next, we introduce the  {\it Stummel-Kato}  class of locally integrable functions for $1<p<d$.
\begin{defin}[{\it Stummel-Kato class}]\label{kato_1}{\em 
	Let $1<p<d$. We say that $V\in L^1_{\loc}(\Gw)$ belongs to the class $\bar{K}^p_{\loc}(\Omega)$ if for any $\omega\Subset \Gw$
	$$ \eta_{V,\gw}(r) :=\sup_{x\in \omega}\int_{\omega\cap B_r(x)}\frac{|V(y)|}{|x-y|^{d-p}}\d y <\infty \qquad \forall\,  0<r<\diam(\omega).$$
	A function $V\in \bar{K}^p_{\loc}(\Omega)$ is said to belong to the local {\it Stummel-Kato class} $K^p_{\loc}(\Omega)$ if for every $\omega\Subset\Gw$  we have   $\displaystyle{\lim_{r\rightarrow 0} \eta_{V,\gw}(r)=0}$.	(see for example, \cite[Definition 1.1]{Difazio} for $p=2$, and \cite[Definition 2.1]{ragusa_zamboni} for $1<p<d$).  If $\Gw=\R^d$ we write $\eta_{V}(r):= \eta_{V,\R^d}(r)$.
}
\end{defin}
We call $\eta_{V,\gw}(r)$  the {\it Stummel-Kato modulus} of $V$ in $\gw$. It is known that  $\eta_{V,\gw}(r)$ is a nondecreasing continuous function of $r$ with the doubling property \cite{ragusa_zamboni}, i.e., there exists constant $\mathcal{K} \in(0,1)$ such that 
$$\mathcal{K}\eta_{V,\gw}(2r)\leq \eta_{V,\gw}(r) \qquad \forall\, r>0.$$
The constant $\mathcal{K}$ is called the doubling constant of $\eta_{V,\gw}$. The Stummel-Kato class $K^p_{\loc}(\Omega)$ is equipped with the  seminorms given by 
$$\|V\|_{K^p(\omega)}:= \eta_{V,\gw}(\diam(\omega)),$$ 
for every $\omega\Subset\Omega$, and therefore, $K^p_{\loc}(\Omega)$ is a Fr\'echet space. In fact, for any $\omega\Subset\Gw$ and a fixed $0<r<\diam(\omega)$, $\eta_{V,\gw}(r)$ defines a semi-norm in $K^p_{\loc} (\Omega)$ and $\eta_{V,\gw}(r)\leq \|V\|_{K^p(\omega)}$.

When $p=2$, $K^2_{\loc}(\Gw)$ is the classical Kato class which was introduced by Kato \cite{kato} to study the self-adjointness of Schr\"odinger operators. Aizenman and Simon \cite{AS} for $A=I$, and  later Chiarenza et al. \cite{chiarenza et al} established  for $p=2$ a local Harnack inequality and the continuity of solutions for the Schr\"odinger-type equation $-\mathrm{div}(A\nabla u)+Vu=0$, where $A$ is a symmetric matrix such that $A\in L^\infty_{\loc}(\R^d)$  is locally uniformly elliptic, and $V$
belongs to the {\it Stummel-Kato} class. For $1<p<d$, the higher order {\it Stummel-Kato} class potentials can be traced back to Davies and Hinz work \cite{davies_hinz}. In general, the study of higher-order Schr\"odinger operators with potentials in {\it Stummel-Kato class} $K^p_{\loc}(\Omega)$ turns out to be more difficult than in the case $p=2$. More details about the {\it Stummel-Kato} class and related theory in PDEs can be found for example in Rasgusa and Zamboni \cite{ragusa_zamboni}, Zamboni \cite{zamboni}, Zheng and Yao \cite{zheg_yao}, and in references therein.

\medskip

Lemma \ref{lem1} below asserts that for $1<p<d$, the local Morrey space $M^q_{\loc}(p;\Omega)$ is a (proper) subset of $K^p_{\loc}(\Omega)$. For a proof see \cite[Lemma 1.29]{maly_ziemer}.
\begin{lem}\label{lem1}
	Let $1<p<d$ and $f\in M^q_{\loc}(p;\Omega)$ with $q>d/p$. Then there exists a constant $C>0$ such that for $0<r<\diam(\omega)$
	$$\int_{\omega\cap B_r(x)} \frac{|f(y)|}{|x-y|^{d-p}}\d y \leq C r^{p-d/q} \|f\|_{M^q(\omega)}.$$
\end{lem}
We next give a subclass of $K^p_{\loc}(\Omega)$ in which each {\it Stummel-Kato modulus} $\eta_{V,\gw}(r)$ of $V\in K^p_{\loc}(\Omega)$ satisfies a particular integral condition near zero.  
\begin{defin}[{See \cite{ragusa_zamboni, zamboni}}]{\em
		Let $1<p<d$,  $\sigma \in (0,1)$, and let $\theta := {p}/{(\sigma p'+p)}$, so, $0<1-\theta=\frac{\sigma}{\sigma+p-1}<1$. A function $V\in K^p_{\loc}(\Omega)$ is said to belong to the class $\tilde{K}^p_{\loc}(\Omega)$ if there exists $\delta>0$ such that 
		$$\int_0^\delta \frac{1}{t}\left(\int_0^t \frac{\eta_V^{1-\theta}(s)}{s}\ds\right)^{\frac{1}{p}}\d t <\infty.$$ 
	}
\end{defin}
\begin{rem}\label{remrk}{\em
 Clearly, $\tilde{K}^p_{\loc}(\Omega)\subset K^p_{\loc}(\Omega)\subset \bar{K}^p_{\loc}(\Omega)$. Also, it can be seen that for $1<p<d$ the local Morrey space $M^q_{\loc}(p;\Omega)$ is in fact, a proper subspace of the class $\tilde{K}^p_{\loc}(\Omega)$. Indeed, if $f\in M^q(p;\omega)$, then by Lemma \ref{lem1}, we have $\eta_{V}^{1-\theta}(s)\leq C s^{(p-d/q)(1-\theta)} \|f\|^{1-\theta}_{M^q(\omega)}$. Hence for $\delta>0$,
 \begin{align*}
\int_0^\delta \frac{1}{t}& \left(\int_0^t \frac{\eta_V^{1-\theta}(s)}{s}\ds\right)^{\frac{1}{p}}\!\!\dt\leq C\|f\|^{(1-\theta)/p}_{M^q(\omega)}\int_0^\delta \frac{1}{t}\left(\int_0^t s^{(p-d/q)(1-\theta)-1}ds\right)^{1/p}\!\!\!\!\dt\\
&= C\|f\|^{(1-\theta)/p}_{M^q(\omega)}\int_0^\delta t^{\frac{(p-d/q)(1-\theta)}{p}-1}\d t =C\|f\|^{(1-\theta)/p}_{M^q(\omega)} \delta^{\frac{(p-d/q)(1-\theta)}{p}}<\infty.
 \end{align*}}
\end{rem}
The next result is an uncertainty-type inequality due to Ragusa and Zamboni \cite[Corollaries 2.6, 2.7]{ragusa_zamboni} (see also \cite[Theorem 2.1, Corollary 2.2]{zamboni}).
\begin{thm}\label{thm_uncrtn}
	Let $1<p<d$, $0<\theta<1$ as above, and $\omega \Subset \mathbb{R}^d$.  Suppose that a function $V\in K^{p}(\omega)$ satisfies $$\int_0^\rho \frac{\eta_V^{1-\theta}(s)}{s}\ds<\infty\qquad \mbox{for some  $\rho> 0$},$$ and let 
	$$\Phi(r): = C(d) \int_0^r\frac{\eta_V^{1-\theta}(t)}{t}\d t.$$
	 Then 
	\begin{itemize}
		\item[(i).] There exists a constant $C(d,p)$ such that 
		$$\int_{\omega\cap B_r} |V(x)||u(x)|^p\dx \leq C(d, p) \Phi(r) \int_{\omega\cap B_r}|\nabla u(x)|^p\dx$$
		for every $u\in C_c^\infty(w)$ with $\supp (u)\subseteq \omega\cap B_r$. 
		\item[(ii).] For any $\delta >0$ the following inequality holds true for any  $u\in C_c^\infty(\omega)$
		$$\int_\omega |V(x)||u(x)|^p\dx \leq \delta \int_\omega |\nabla u(x)|^p \dx + \frac{C_1(\diam(\omega))\delta}{\left[\Phi^{-1}\left(\frac{\delta}{C(d,p,\gh_V)}\right)\right]^{d+p}} \int_\omega |u(x)|^p \dx.$$
	\end{itemize}	 
\end{thm}
Using a standard density argument one can replace in Theorem~\ref{thm_uncrtn} the space $C_c^\infty(\omega)$ with $W_0^{1,p}(\omega)$. Theorem~\ref{thm_uncrtn} is an extension of the Morrey-Adams theorem to the Kato class (see for example \cite{pinchover_psaradakis, giri_pinchover} and references therein, where the potential $V$ is assumed to be in the local Morrey space $M^q_{\loc}(p;\Omega)$). 
Using Theorem~\ref{thm_uncrtn} and the Moser iteration technique, Ragusa and Zamboni \cite{ragusa_zamboni} proved the local boundedness of weak solutions and the validity of Harnack inequality for nonnegative weak solutions of the equation $-\Delta_p u+ V|u|^{p-2}u=0$ for  $V\in \tilde{K}^p_{\loc}(\Omega)$. Hence, as a consequence of the Harnack inequality, the local continuity of weak solutions is obtained (see \cite[Theorem 6.1]{ragusa_zamboni}). Note that if the potential $V$ lies in the Morrey space $M^q_{\loc}(p;\Omega)$ with $1<p<d$, then weak solutions are in fact H\"{o}lder continuous (see Mal\'{y} and Ziemer \cite[Theorem 4.11]{maly_ziemer} and also Di Fazio\cite{Difazio}).  

\medskip

\subsection{Wolff class} In this subsection, we introduce another space of locally integrable functions for $1<p<d$ which is known as the {\it Wolff class}.
\begin{defin}[Wolff class]\label{wolff_cls}{\em 
		Let $1<p<d$. We say that $f\in L^1_{\loc}(\Omega)$ belongs to the {\em Wolff class} $\mathfrak{W}_{\loc}^p(\Gw)$ if for every $\omega\Subset \Gw$ 
		 $$ \lim_{r\rightarrow 0}W_f(r):=\lim_{r\rightarrow 0}\left(\sup_{x\in \omega}\int_0^r\left[\frac{1}{s^{d-p}}\int_{\omega\cap B_s(x)}|f(y)|\d y\right]^{\frac{1}{p-1}} \!\frac{\ds}{s}\right)=0.$$	
	}
\end{defin}
We call $W_f(r)$ {\em the Wolff modulus of} $f$. Note that we should consider $d\geq 2$ when $1<p<2$, and $d\geq 3$ when $p=2$. For $p=2$, by \cite[Lemma 1.27]{maly_ziemer} we have 
$$\int_0^r\frac{1}{s^{d-1}}\left(\int_{\omega\cap B_s(x)}|f(y)|\dy\right) \ds=\frac{1}{(d-2)}\int_{\omega\cap B_r(x)} \frac{|f(y)|}{|x-y|^{d-2}}\dy.$$
Thus,  $\mathfrak{W}^2_{\loc}(\Gw)=K^2_{\loc}(\Gw)$, i.e., for $p=2$ the Wolff space $\mathfrak{W}_{\loc}^p(\Gw)$ is in fact the standard Kato class. The following remark gives a relationship between the Kato class $K^p_{\loc}(\Omega)$ and the Wolff class $\mathfrak{W}^p_{\loc}(\Omega)$ for $p\in(1,d)$.
\begin{remark}\label{remrk1}{\em
	Let $1<p<d$ and $r>0$ be sufficiently small. Then
		\begin{align*}
		W_f(r)&\!=\!\sup_{x\in \omega}\int_0^r\left[\frac{1}{s^{d-p}}\int_{\omega\cap B_s(x)}|f(y)|\d y\right]^{\frac{1}{p-1}} \frac{\d s}{s}\\ 
		&\!\leq \! \sup_{x\in \omega}\!\int_0^r\!\!\left[\!\int_{\omega\cap B_s(x)}\!\!\frac{|f(y)|}{|x\!-\!y|^{d-p}}\!\dy\!\right]^{\frac{1}{p-1}} \!\!\!\frac{\ds}{s}
		 \!\leq\!   \int_0^r \!\frac{ \eta_{f,\gw}^{\frac{1}{p-1}}(s) }{s}\!\ds \!\leq\! \int_0^r\! \frac{\eta_{f,\gw}^{1-\theta}(s)}{s}\!\ds, 
		\end{align*}
	since $$1-\theta =\frac{\sigma}{\sigma+p-1}<\frac{1}{p-1}\,\,\,\,\,\mbox{with } 0<\sigma<1.$$
		This implies that if $f\in K_{\loc}^p(\Omega)$  and  for every $\omega\Subset \Gw$ there exists $\delta>0$ such that $\int_0^\delta \frac{\eta_{f,\gw}^{1-\theta}(s)}{s}\d s <\infty$, then $f$ is in the Wolff class $\mathfrak{W}_{\loc}^p(\Gw)$.
	}
\end{remark}
\begin{lemma}\label{lem_banach}
 For $1<p\leq 2$ and $0<r\leq \diam(\omega)$,  the expression 
$$\|f\|_{\mathfrak{W}^p(\omega)}:= W_f^{p-1}(r)=\left[\sup_{x\in \omega}\int_0^r\left[\frac{1}{s^{d-p}}\int_{\omega\cap B_s(x)}|f(y)|\d y\right]^{\frac{1}{p-1}} \frac{\d s}{s}\right]^{p-1}$$ 
defines a norm on the Wolff class $\mathfrak{W}^p(\omega)$. Moreover,   $(\mathfrak{W}^p(\omega), \|\cdot\|_{\mathfrak{W}^p(\omega)})$ is a Banach space.
\end{lemma}
 \begin{proof}
 We will show only the triangle inequality. Since $1<p\leq 2$ implies $\frac{1}{p-1}\geq 1$, the Minkowski inequality implies that  
\begin{align*}
&\|f_1+f_2\|_{\mathfrak{W}^p(\omega)} = \left[\sup_{x\in \omega}\int_0^r\left[\frac{1}{s^{d-p}}\int_{\omega\cap B_s(x)}|f_1(y)+f_2(y)|\d y\right]^{\frac{1}{p-1}} \frac{\d s}{s}\right]^{p-1}\\
&\leq  \sup_{x\in \omega}\left[\int_0^r\left[\frac{1}{s^{d-1}}\int_{\omega\cap B_s(x)}|f_1(y)|\d y+ \frac{1}{s^{d-1}}\int_{\omega\cap B_s(x)}|f_2(y)|\d y\right]^{\frac{1}{p-1}} \d s\right]^{p-1}\\
& \leq \sup_{x\in \omega}\left[\int_0^r\left[\frac{1}{s^{d-1}}\int_{\omega\cap B_s(x)}|f_1(y)|\d y\right]^{\frac{1}{p-1}}\!\!\ds\right]^{p-1} \\
& + \sup_{x\in \omega}\left[\int_0^r\left[\frac{1}{s^{d-1}}\int_{\omega\cap B_s(x)}|f_2(y)|\d y\right]^{\frac{1}{p-1}}\!\!\!\d s\right]^{p-1}= \|f_1\|_{\mathfrak{W}^p(\omega)} + \|f_2\|_{\mathfrak{W}^p(\omega)}.
\end{align*}
Following \cite[Lemma 2.7]{biroli_mosco}, we infer that the Wolff space $\mathfrak{W}^p(\omega)$ with the norm $\|\cdot\|_{\mathfrak{W}^p(\omega)}$ is a Banach space for $1<p\leq 2$. 
 \end{proof}
\begin{lemma}
	For $p> 2$, the expression 
	$$\|f\|_{\mathfrak{W}^p(\omega)}:= W_f^{p-1}(r)=\left[\sup_{x\in \omega}\int_0^r\left[\frac{1}{s^{d-p}}\int_{\omega\cap B_s(x)}|f(y)|\d y\right]^{\frac{1}{p-1}} \frac{\d s}{s}\right]^{p-1}$$ 
	defines a quasinorm on $\mathfrak{W}^p(\omega)$ for every fixed $0<r\leq \diam(\omega)$.
\end{lemma}
\begin{proof}
If $p>2$, then $0<\frac{1}{p-1}<1$. Using the same argument as in the proof of Lemma~\ref{lem_banach},  we obtain 
\begin{equation}
\|f_1+f_2\|_{\mathfrak{W}^p(\omega)}\leq 2^{p-2}\bigg(\|f_1\|_{\mathfrak{W}_p(\omega)} + \|f_2\|_{\mathfrak{W}^p(\omega)}\bigg),\label{eqn_1}
\end{equation}
where we use the following two inequalities with $a,b>0$: $(a+b)^\lambda\leq a^\lambda + b^\lambda\,\,\,  \mbox{for }\, 0<\lambda <1$ and
$(a+b)^\lambda\leq 2^{\lambda-1}(a^\lambda + b^\lambda)\,\,\,  \mbox{for }\, \lambda >1$. 
\end{proof}
The following  crucial Morrey-Adams type estimates for  potentials in the Wolff class $\mathfrak{W}_{\loc}^p(\Gw)$ extend  Theorem~\ref{thm_uncrtn}, and will be used in our study. 
\begin{theo}\label{morrey_adams}
	Let $1<p<d$, $\omega \Subset \Gw$, and $V\in \mathfrak{W}^p(\omega)$. 
	\begin{itemize}
		\item[(i).] For any $\epsilon>0$ there exist $0<r_0<1$ and $\tau_\epsilon>0$ (depending on $\epsilon, d, p$) such that  if $B_{4r_0}(x_0) \subset \omega$ and $W_V(r_0)<\tau_\epsilon$, then 
		$$\int_{B_r(x_0)}|V||u|^p \dx \leq \epsilon \int_{B_r(x_0)} |\nabla u|^p \dx \qquad  \forall r\leq r_0 \mbox{ and }	 u\in W_0^{1,p}(B_r(x_0)).$$
		\item[(ii).]  Let $r>0$ and $B_{2r}(x_0)\subset \omega$.  Then there exists a constant $C>0$ depending on $d,p$ such that 
		$$\int_{B_r(x_0)}|V||u|^p \dx \leq C W^{p-1}_V(2r) \int_{B_r(x_0)} |\nabla u|^p \dx \qquad \forall u\in W_0^{1,p}(B_r(x_0)).$$
	\end{itemize}
\end{theo}
\begin{proof}
	See \cite[Lemma 2.1]{LS} and \cite[Lemma 2.1]{Skrypnik}.
\end{proof}
As a consequence of Theorem \ref{morrey_adams} (ii), we obtain the following two Morrey-Adams type theorems in $\gw \Subset \Gw$.
\begin{theo}\label{cor_1}
	Let $1<p<d$ and $V\in \mathfrak{W}^p_{\loc}(\Gw)$. Then for any $\omega'\Subset\omega\Subset \Gw$ 
	and $\gd \! > \! 0$, there exists a positive constant $C_\gd \!  =  \! C(d, p, \delta, \omega', \omega, W_V^p(\diam(\omega)))$ such  that
	\begin{equation*}\label{eqn1}
	\int_{\omega'}|V||u|^p\dx\leq \delta \|\nabla u\|^p_{L^p(\omega', \mathbb{R}^d)} + C_\gd\| u\|^p_{L^p(\omega')}  \qquad \forall u\in W_0^{1,p}(\omega').
	\end{equation*}
\end{theo}
\begin{proof}
	The proof follows straight forwardly from Theorem~\ref{morrey_adams} by using a partition of unity of $\omega'$ (see \cite[Corollary 2.7.]{ragusa_zamboni} for a similar proof).
\end{proof}
\begin{theo}[Morrey-Adams type theorem]\label{cor_3}
	Let $1<p<d$ and $V\in \mathfrak{W}^p_{\loc}(\Gw)$. Then for any $\omega'\!\Subset\! \omega \!\Subset\! \tilde{\omega}\Subset \Gw$ with $\partial \omega'$ being Lipschitz and $\delta\!>\!0$, there exists $C_\gd \!=\! C(d, p, \delta, \omega', \omega, \tilde{\omega}, W_V(\diam(\tilde{\omega})))>0$ such that
	\begin{equation}\label{eqn2}
	\int_{\omega'}|V||u|^p\dx\leq \delta \|\nabla u\|^p_{L^p(\omega', \mathbb{R}^d)} + C_\gd\|u\|^p_{L^p(\omega')} \qquad \forall u\in W^{1,p}(\omega') 
	\end{equation}
\end{theo}
\begin{proof}
	To prove \eqref{eqn2}, we use an extension operator  $E: W^{1,p}(\omega')\rightarrow W_0^{1,p}(\omega)$ such that for any $u\in W^{1,p}(\omega')$ 
	\begin{equation}\label{eqn_2}
	\begin{cases}
	      E(u)=u \quad \mbox{in } \omega'  ,\\
	        \|E(u)\|_{L^p(\omega)} \leq  C(d, p, \omega', \omega) \|u\|_{L^p(\omega')},\\
	        \|\nabla E(u)\|_{L^p(\omega,\mathbb{R}^d)} \leq C(d, p, \omega', \omega) \|\nabla u\|_{L^p(\omega',\mathbb{R}^d)}.
	\end{cases}
	\end{equation}
	Then applying Theorem~\ref{cor_1} in $\gw$, we have
	\begin{align*}
	&\int_{\omega'}|V||u|^p\dx\leq  \int_{\omega}|V||E(u)|^p\dx\\[2mm]
	& \leq  \delta \|\nabla E(u)\|^p_{L^p(\omega, \mathbb{R}^d)} + C(d, p, \delta, \omega, \tilde{\omega}, W_V(\diam(\tilde{\omega})))\|E(u)\|^p_{L^p(\omega)}.
	\end{align*}
	Therefore, by using \eqref{eqn_2} in the latter term, \eqref{eqn2} follows.
\end{proof}
We conclude the present section by the following corollary which demonstrates that the Wolff class  $\mathfrak{W}^p_{\loc}(\Gw)$ is continuously embedded in $L^1_{\loc}(\Gw)$.
\begin{corollary}\label{cor_2}
	Let $1<p<d$ and $V\in \mathfrak{W}^p_{\loc}(\Gw)$. Suppose $x_0 \in \Gw$ and $B_{2r}(x_0)\Subset \Gw$. Then there exists $C=(C(d,p))$ such that
	\begin{equation*}
	\int_{B_{r/2}(x_0)}|V|\dx\leq C r^{d-p} W_V^{p-1}(2r).
	\end{equation*}
\end{corollary}
\begin{proof}
	The proof is immediate by choosing $u\in C_c^\infty(B_{2r}(x_0))$ with $\supp{u}\subset B_{r}(x_0)$, $0\leq u\leq 1$ in $B_r(x_0)$, $u=1$ in $B_{r/2}(x_0)$ and $|\nabla u|<C/r$, in part (ii) of Theorem~\ref{morrey_adams}. 
\end{proof}

\section{Criticality theory for $Q_{p,A,V}$ with potential in Wolff class}\label{pde_wolff}
The present section is devoted to criticality theory for weak positive solutions of the equation \eqref{plaplace} with potentials in the local Wolff class  $\mathfrak{W}^p_{\loc}(\Omega)$. 
\subsection{The setting} 
Let $\Gw \subseteq \R^d$ be a domain, and fix  $1<p<d$. Throughout the paper we assume that the operator $Q=Q_{p,A,V}$ satisfies the following local regularity assumptions.  
\begin{assumptions} \label{assump} {\em 
		\begin{itemize}
			\item[{\ }]		
			\item $A \!=\! (a_{ij})_{i,j=1}^{d} \!\in\! L_{\rm loc}^{\infty}(\Gw;\R^{d\times d})$ is a symmetric matrix valued function.
			 \item $A$ is locally uniformly elliptic  in $\Gw$, that is, for any  $\gw\Subset \Gw$ there exists  $\Theta_{\gw}>0$ such that for all $\xi\in \mathbb{R}^d$ and $  x\in \gw$
			\begin{eqnarray*} 
			\hspace*{1.5cm}(\Theta_{\gw})^{-1}\sum_{i=1}^d\xi_i^2
			\leq |\xi |_{A(x)}^2:=\sum_{i,j=1}^d
			a_{ij}(x)\xi_i\xi_j
			\leq \Theta_{\gw}\sum_{i=1}^d\xi_i^2 .
			\end{eqnarray*}
			\item  $V\in\mathfrak{W}_{\loc}^p(\Gw)$ is a real valued function.
		\end{itemize} }
\end{assumptions}
By Assumptions~\ref{assump}, the function $\xi\mapsto |\xi|_{A(x)}^p$ is {\em strictly convex} a.e. $x\in \Omega$. At some points, we assume that $ |\xi|_A^p$ is locally strong convexity.
\begin{assumption}\label{assump_strong}
{\em	We say that $|\xi|_A^p$ is a {\em locally strong convex function} with respect to $\xi\in \mathbb{R}^d$ if there exists $\bar{p}\geq p$ such that for every $\omega\Subset \Omega$ there exists a positive constant $C_\omega(\bar{p})$ such that 
	\begin{equation*}\label{eq_pbar}
	|\xi|_A^p-|\eta|_A^p-p|\eta|_A^{p-2}A(x)\eta\cdot (\xi-\eta)\geq C_\omega(\bar{p})|\xi-\eta|_A^{\bar{p}}\quad \forall \,\xi, \eta \in \mathbb{R}^d \mbox{ and a.e. } x\in \omega.
	\end{equation*}
}
\end{assumption}
	In fact, by \cite[Lemma 3.4]{PR} and \cite[Lemma 2.2]{PTT}, there exists a constant  $C(p)>0$ such that for  all $\xi,\eta\in\mathbb{R}^{n}$ ($\eta\neq 0$ if~$p<2$) and a.e.~$x\in\omega$,
\begin{equation}\label{eq:strong-conv}
|\xi|^{p}_{A}-|\eta|^{p}_{A}-p|\eta|^{p-2}_{A}A(x)\eta\cdot(\xi-\eta)
\geq C(p)[\xi,\eta]_{p,A},
\end{equation}
where
\[[\xi,\eta]_{p,A}\triangleq\begin{cases}
|\xi-\eta|_{A}^{p}&\text{if~$p\geq 2$},\\
|\xi-\eta|_{A}^{2}(|\xi|_{A}+|\eta|_{A})^{p-2}& \text{if~$1<p<2$}.
\end{cases}\]
\begin{definition}{\em 
A function $v$ is said to be a {\em (weak) solution} of the equation $Q_{p,A,V}(u)=0$ in $\Gw$  if  $v\in W_{\loc}^{1,p}(\Omega)$ and $v$ satisfies  
\begin{equation}\label{eq-ws}
\int_\Omega(|\nabla v|^{p-2}_A A\nabla v\cdot \nabla \varphi+ V|v|^{p-2}v\varphi)\dx=0\qquad \forall \varphi\in C_c^\infty(\Omega).
\end{equation}
	
 We say the $v\in W_{\loc}^{1,p}(\Omega)$ is a {\em supersolution} of  \eqref{plaplace} in $\Gw$, if the integral in \eqref{eq-ws} is nonnegative for every nonnegative test function $\varphi\in \C_c^\infty(\Omega)$. A function $v$ is a {\em subsolution} of  \eqref{plaplace} if $-v$ is supersolution of \eqref{plaplace}. }
\end{definition}
 Theorem~\ref{morrey_adams} implies that the integral in  \eqref{eq-ws} is well defined.  Note that \eqref{plaplace} is the Euler-Lagrange equation associated with the energy functional 
\begin{equation*}\label{euler-lagrange eq}
	\mathcal{Q}_{p,A,V}(\vgf)= \mathcal{Q}_{p,A,V}(\vgf; \Gw):=\int_\Omega (|\nabla \vgf|^p_A+V|\vgf|^p)\dx \qquad\vgf\in C_c^\infty(\Omega).
\end{equation*}
Throughout the paper we write  $\mathcal{Q}_{p,A,V}\geq 0$ in $\Gw$ if  $\mathcal{Q}_{p,A,V}(\vgf)\geq 0$  for all $\vgf\in C_c^\infty(\Omega)$.  
\subsection{Harnack convergence principle}
In this subsection we prove the Harnack convergence principle. First we recall the local Harnack inequality for nonnegative weak solutions of  the equation $Q_{p,A,V}(u)=0$ for $1<p<d$ proved in\cite[Theorem 2.2]{biroli_1}. We also refer to \cite[Theorem 2.2]{skrypnik_1}, where the local Harnack inequality of nonnegative solutions is proved for certain quasilinear equations with coefficients in Wolff class. The Harnack inequality of positive solutions is further extended in \cite[Theorem 1.2]{LS_1} for a certain class of divergence-type elliptic equations with nonstandard growth conditions and coefficients from appropriate Wolff classes.
\begin{thm}[Local Harnack inequality]\label{thm_hrnck}
	Let $1<p<d$ and $\omega'\Subset\omega\Subset \Gw$. Then under Assumptions~\ref{assump}, any nonnegative solution $v$ of the equation $Q_{p,A,V}(u)=0$ in $\gw$ satisfies  
	\begin{equation*}\label{h_ineq}
	\sup_{\omega'} v \leq C\inf_{\omega'} v,
	\end{equation*}
	where $C$ is a positive constant depending only $d, p, \Theta_\omega$,   $W_V(\diam(\omega))$ and $\dist{\omega'}{\omega} $, but not on $v$.
\end{thm}
We recall the weak Harnack inequality which is valid for nonnegative supersolutions of  $Q_{p,A,V}(u)=0$. For the proof see \cite[Theorem 2.1]{skrypnik_1}.
\begin{theorem}[Weak Harnack inequality]\label{w_hrnck}
	Let $1<p<d$ and $\omega'\Subset \omega\Subset \Gw$. Then under Assumptions~\ref{assump}, for every $0<q< \frac{d(p-1)}{d-p}$ any nonnegative supersolution $v$ of  the equation of  $Q_{p,A,V}(u)=0$ in $\gw$ satisfies 
	$$\|v\|_{L^q(\omega')} \leq C \inf_{\omega'} v,$$
	 where the constant $C>0$ depends only on $d$, $p$, $q$, $\Theta_\omega$, $\mathcal{L}^d(\omega')$, $\dist{\omega'}{\omega}$ and $W_V(\diam(\omega))$. Here $ \mathcal{L}^d(\omega')$ is the Lebesgue measure of $\gw'$.
\end{theorem}
As a consequence of Theorem~\ref{thm_hrnck}, we have the following result.
\begin{cor}[Local continuity]\label{mdlus_cont}
Let $1<p<d$. Then under Assumptions~\ref{assump}, any solution $v$ of \eqref{plaplace} is continuous and has the following bound on its modulus of continuity: if $B_{r_0}:=B_{r_0}(x_0)\Subset\Omega$	and $\beta \in (0,1)$, then there exist $\alpha \in(0,1)$ depending on $\beta$, and positive constants $C=C\big(d,p,\Theta_{B_{r_0}}, W_V(r_0)\big)$, $\gamma=\gamma \big(\d,p,\Theta_{B_{r_0}}, W_V(r_0)\big)$, such that 
$$\underset{B_r}{\osc} \,v \leq C\left[\left(\frac{r}{r_0}\right)^\alpha \underset{B_{r_0}}{\osc} \,v + \gamma W_V(r_0^{1-\beta}r^\beta) \right]\quad \forall\, 0<r<r_0. $$
\end{cor}
\begin{proof}
 Since the solution $v$ of \eqref{plaplace} is locally bounded \cite[Theorem 3.5]{biroli_1}, hence  $\|v\|_{L^\infty(B_{r_0})} <\infty$. For $0<r< r_0$, define 
$$M(r):= \sup_{B_r} v,\quad m(r):= \inf_{B_r} v,\quad \tau(r):=\underset{B_r}{\osc}\, v= M(r)- m(r).$$
 Following the arguments in \cite[Theorem 4.11]{maly_ziemer} or in\cite[Theorem 8.22]{trudinger}, one obtains that there exists a positive constant $\vartheta=\vartheta(d,p,\Theta_{B_{r_0}}, W_V(r_0))\in(0,1)$,   such that 
$$\tau\bigg(\frac{r}{2}\bigg) \leq \vartheta \tau(r) + \gamma W_V(r) \qquad \forall\, 0<r < r_0,$$
where $\gamma>0$ is a constant depending on \,$\d,p,\Theta_{B_{r_0}}, W_V(r_0)$ (see \cite[Corollary 1.5.]{LS_1} for $p(x)\equiv p$).


Then by applying \cite[Lemma 4.12]{maly_ziemer}, or \cite[Lemma 8.23]{trudinger}, it follows that for any $\beta\in(0,1)$, there exist a positive constant $C=C(\vartheta)>0$ and $\alpha \in(0,1)$ such that 
$$ \tau(r)\leq C\bigg[\left(\frac{r}{r_0}\right)^\alpha \tau(r_0) + \gamma W_V(r_0^{1-\beta}r^\beta)\bigg] 
\qquad \forall\, 0<r<r_0. $$
This implies the continuity of solutions since $W_V(r)\rightarrow 0$ as $r\rightarrow 0$.
\end{proof}
\begin{remark}{\em	
		When the potential $V$ lies in the Wolff class $\mathfrak{W}_{\loc}^p(\Gw)$  or in the Stummel-Kato class $\tilde{K}^p_{\loc}(\Omega)$, one should not  expect  H\"older continuity of solutions of \eqref{plaplace} (see for example, \cite{chiarenza et al,ragusa_zamboni}). 
}
\end{remark}
The following version of the  Harnack convergence principle will be used several times throughout the paper.
\begin{theorem}[Harnack convergence principle]\label{hrnck_principle}
Let $\{\Gw_i\}$	be a compact exhaustion of the domain $\Gw$. Suppose that $\{A_i\} \subset L^\infty(\Gw_i;\mathbb{R}^{d\times d})$ is a sequence of symmetric and locally uniformly positive definite matrices such that their local ellipticity constants are uniformly bounded away from $0$, and $\{A_i\}$ converges weakly in $L^\infty_{\loc}(\Gw;\mathbb{R}^{d\times d})$ to  $\mathbb{A} \in L^\infty_{\loc}(\Gw;\mathbb{R}^{d\times d})$ satisfying Assumptions~\ref{assump}. Furthermore, assume that 
\begin{itemize}
	\item[(i)] for $1<p\leq 2$, $\{V_i\}\in \mathfrak{W}^p(\Gw_i)$ converges weakly in the Wolff class $\mathfrak{W}_{\loc}^p(\Gw)$ to $\mathbb{V}\in \mathfrak{W}_{\loc}^p(\Gw)$,
	\item[(ii)] for $p>2$,  $\{V_i\}\in \mathfrak{W}^p(\Gw_i)$ converges with respect to the quasinorm of the Wolff class $\mathfrak{W}_{\loc}^p(\Gw)$ to $\mathbb{V}\in \mathfrak{W}_{\loc}^p(\Gw)$.
\end{itemize}

For each $i\in\mathbb{N}$, let $v_i$ be a positive (continuous) solution of  $Q_{p,A_i, V_i}(v)=0$ in $\Gw_i$ with $v_i(x_0)=1$, where $x_0\in \Gw_1$ is a fixed reference point. 

Then, up to a subsequence, $\{v_i\}$ converges weakly in $W^{1,p}_{\loc}(\Gw)$ and also locally uniformly to a positive solution $v$ of the equation $Q_{p,\mathbb{A}, \mathbb{V}}(v)=0$ in $\Gw$.
\end{theorem}
\begin{proof}
	(i) Assume first that  $1<p\leq 2$.  So,  $\{V_i\}\in \mathfrak{W}^p(\Gw_i)$ converges weakly in the Wolff class $\mathfrak{W}_{\loc}^p(\Gw)$ to $\mathbb{V}\in \mathfrak{W}_{\loc}^p(\Gw)$. The sequence $\{A_i\}$ of matrices is locally uniformly elliptic and converges in the weak topology of $L^\infty_{\loc}(\Gw;\mathbb{R}^{d\times d})$, so $\|A_i\|_{L^\infty(\tilde{\omega};\mathbb{R}^{d\times d})}\leq C$ for every $\tilde{\omega}\Subset\Gw$. Therefore, the matrices $A_i$ are locally uniformly positive definite matrices which are uniformly bounded in every $\tilde{\omega}\Subset\Gw$ except for a set of measure zero.
	
	Since  $v_i$ is a positive weak solution of the equation $Q_{p,A_i, V_i}(v)=0$ in $\Gw_i$, we have 
	\begin{equation*}\label{heqn_1}
	\int_{\Omega_i}|\nabla v_i|_{A_i}^{p-2}A_i\nabla v_i\cdot\nabla u\dx+\int_{\Omega_i}V_iv_i^{p-1}u\dx=0 \quad \text{for all}\,\,u\in W_0^{1,p}(\Omega_i).
	\end{equation*}
	Fixing a natural number $k$ and $u\in C_c^\infty(\Gw_k)$, we may plug $v_i|u|^p\in W_0^{1,p}(\Gw_k)$, $i\geq k$, as a test function in the above equality to obtain
	\begin{equation*}
	\|\, |\nabla v_i|_{A_i}u\|^p_{L^p(\Omega_k)}\leq p \int_{\Omega_k}|\nabla v_i|_{A_i}^{p-1}|u|^{p-1}v_i|\nabla u|_{A_i}\dx+\int_{\Omega_k}|V_i|v_i^p|u|^p\dx.
	\end{equation*}
Applying Young's inequality  $p a b\!\leq\! \epsilon a^{p'}\!+\![(p\!-\!1)/\epsilon]^{p-1}b^p$, $\epsilon \!\in \! (0,1)$, with $a \! = \! |\nabla v_i|_{A_i}^{p-1}|u|^{p-1}$ and $b\!=\!v_i|\nabla u|_{A_i}$ on the first term, and Morrey-Adams theorem (Theorem~\ref{cor_1}) for the second term of the above inequality, we get
\begin{align*}
&(1-\epsilon)\||\nabla v_i|_{A_i}u\|^p_{L^p(\Omega_k)}\\
&\leq \left(\frac{p-1}{\epsilon}\right)^{p-1} \!\!\|v_i|\nabla u|_{A_i}\|^p_{L^p(\Omega_k)} +\delta\|\nabla(v_iu)\|^p_{L^p(\Omega_k;\mathbb{R}^d)}
 +C_\gd \|v_i u\|^p_{L^p(\Omega_k)},
\end{align*}
where $C_\gd=C\big(d,p,\delta,W_{V_i}(\diam(\Omega_{k+1}))\big)$.
Since the sequence $\{A_i\}$ is locally uniformly elliptic and bounded a.e., therefore, using the inequality
$$\|\nabla(v_iu)\|^p_{L^p(\Omega_k;\mathbb{R}^d)}\leq 2^{p-1}\left(\|v_i\nabla u\|^p_{L^p(\Omega_k;\mathbb{R}^d)}+\|u\nabla v_i\|^p_{L^p(\Omega_k;\mathbb{R}^d)}\right),$$
we obtain the following estimates for every $i\geq k$ and any $u\in C_c^\infty(\Omega_k)$
	\begin{align}\nonumber
&\left((1\!-\!\epsilon)C_{\Omega_k}^p\!-\!2^{p-1}\delta C_{\Omega_k}^{-p}\right)\!\||\nabla v_i|u\|^p_{L^p(\Omega_k)}  \\
&\leq \left(\!\!\left(\!\frac{p\!-\!1}{\epsilon}\!\right)^{p\!-\!1}C^{-p}_{\Omega_k}\!+
\!2^{p-1}\delta\!\right)\!\|v_i|\nabla u|\|^p_{L^p(\Omega_k)}\!\!+\!C_\gd \|v_i u\|^p_{L^p(\Omega_k)}.\label{heqn_2}
\end{align}
Pick an arbitrary subdomain $\omega\Subset\Omega$. Without loss of generality, we may assume that $x_0\in\omega$. Now for a subdomain  $\omega'\Subset\Omega$ with $\omega\Subset\omega'$, there exists $k\geq 1$ such that $\omega'\Subset\Omega_k$. Let $0<\delta <(1-\epsilon)2^{1-p}C^{2p}_{\Omega_k}$ and specialize $u\in C_c^\infty(\Omega_k)$ such that
\begin{equation}\label{heqn_3}
\text{supp}\{u\}\subset \omega',\,0\leq u\leq 1,\,\,u=1\,\text{in}\,\omega, \text{ and } |\nabla u|\leq \frac{1}{\text{dist}(\omega,\partial \omega')}\,\,\text{in}\,\,\omega'.
\end{equation}

Our assumption on the weak convergence of $\{V_i\}$ in $\mathfrak{W}_{\loc}^p(\Gw)$ implies that $\{W_{V_i}(\diam(\omega'))\}$ is bounded in $\R$. Hence, by the local Harnack inequality, the sequence of solutions $\{v_i\}$ is bounded  in $L^\infty(\omega)$. Also, by Corollary~\ref{mdlus_cont}, the solutions $\{v_i\}$  have a common modulus of continuity in $\gw$. Thus, Arzel\`a-Ascoli theorem implies that up to a subsequence, $\{v_i\}$ converges uniformly in $\omega$. Moreover, using $u$ as in \eqref{heqn_3} to the estimates \eqref{heqn_2}, we get
	$$\|\nabla v_i\|^p_{L^p(\omega;\mathbb{R}^d)}+\|v_i\|^p_{L^p(\omega)}\leq C\big(d,p,\epsilon,\delta, \text{dist}(\omega,\partial \omega'), C_{\Omega_k}, W_{V_i}(\diam(\Omega_{k+1}))\big),$$
for every $i\geq k$. This shows that the sequence $\{v_i\}$ is bounded in $W^{1,p}(\omega)$. Hence, up to a subsequence, $\{v_i\}$ converges weakly to a nonnegative function $v\in W^{1,p}(\omega)\cap C(\gw)$ with $v(x_0)=1$. Since up to subsequence, $\{v_i\}$ converges uniformly in $\gw$, it follows that
\begin{align*}
v_i \rightarrow v \quad  \text{ uniformly in } \omega,  \quad \text{ and } \nabla v_i  \rightharpoonup \nabla v\,\,\text{in}\,L^{p}(\omega;\mathbb{R}^d).
\end{align*}

{\bf Claim:}  $v$ is a weak solution of $Q_{\mathbb{A},p,\mathbb{V}}(u)=0$ in $\tilde{\omega}\Subset \omega$. Indeed, using the uniform convergence  of  $\{v_i\}$ to $v$ in $\gw$, we get
\begin{align}
&\left|\int_{\tilde{\omega}}(V_iv_i^{p-1}\vgf -\mathbb{V}v^{p-1}\vgf)\dx\right| \nonumber \\
&\leq \int_{\tilde{\omega}} |V_i|\,|v_i^{p-1}-v^{p-1}|\,|\vgf| \dx+\left|\int_{\tilde{\omega}} (V_i -\mathbb{V})v^{p-1}\vgf \dx\right| \nonumber\\
&\leq C\|v_i-v\|_{L^\infty(\tilde{\omega})} \int_{\tilde{\omega}}|V_i|\dx+\left|\int_{\tilde{\omega}} (V_i -\mathbb{V})v^{p-1}\vgf \dx\right| \quad \forall \varphi \in C_c^\infty(\tilde{\omega}).\label{qqn}
\end{align}
 Since $\{W_{V_i}(\diam(\omega'))\}$ is bounded in $\R$, the sequence $\{V_i\}$ is bounded in $L_{\loc}^1(\Omega)$ (see Corollary~\ref{cor_2}). Thus, the first term of the above inequality converges to zero, while the second term converges to zero by our assumption  on the weak convergence of $\{V_i\}$ to $\mathbb{V}$. 
 Therefore, from \eqref{qqn}, we get
\begin{equation*}\label{heqn_4}
\int_{\tilde{\omega}}V_iv_i^{p-1}\vgf\dx \rightarrow \int_{\tilde{\omega}}\mathbb{V}v^{p-1}\vgf\dx \qquad \text{for every } \vgf\in C_c^\infty(\tilde{\omega}).
\end{equation*}
On the other hand, the  sequence $\{A_i\}$ of matrices converges weakly in $L^\infty_{\loc}(\Gw;\mathbb{R}^{d\times d})$ to $\mathbb{A}\in L^\infty_{\loc}(\Gw;\mathbb{R}^{d\times d})$ and $\nabla v_i  \rightharpoonup \nabla v\,\,\text{in}\,L^{p}(\omega;\mathbb{R}^d)$, consequently, by \cite[Proposition 2.7]{giri_pinchover} it follows that
\begin{equation*}\label{heqn_w}
|\nabla v_i|^{p-2}_{A_i}A_i\nabla v_i \rightharpoonup_{i\rightarrow\infty} |\nabla v|^{p-2}_{\mathbb{A}}\mathbb{A}\nabla v \,\,\,\text{in}\,L^{p'}(\tilde{\omega};\mathbb{R}^d).
\end{equation*}
Thus, the claim is proved.\\
(ii) Assume now that  $p>2$. In this case,  we only need to show that 
\begin{equation}\label{heqn_5}
\int_{\tilde{\omega}}V_iv_i^{p-1}\vgf\dx \rightarrow \int_{\tilde{\omega}}\mathbb{V}v^{p-1}\vgf\dx \qquad \text{for every } \vgf\in C_c^\infty(\tilde{\omega}),
\end{equation}
since the rest of the proof follows as in (i). Obviously,  
\begin{align}
&\left|\int_{\tilde{\omega}}\!\!(V_iv_i^{p-1}u \!-\!\mathbb{V}v^{p-1}u)\!\dx\right|
\!\leq \!\int_{\tilde{\omega}}\!\!|V_i \!-\!\mathbb{V}||v_i|^{p-1}\!|u| \!\dx \!+\!\int_{\tilde{\omega}}\!\! |\mathbb{V}||v_i^{p-1} \!-\! v^{p-1}| |u| \!\dx \nonumber\\
&\leq C\!\int_{\tilde{\omega}} \!\! |V_i-\mathbb{V}| |u| \dx+\!\int_{\tilde{\omega}}\! |\mathbb{V}||v_i^{p-1}-v^{p-1}| |u|\dx. \label{eq18}
\end{align}
Clearly,  $\mathbb{V} v_i^{p-1}\rightarrow \mathbb{V}v^{p-1}$ pointwise in $\tilde{\omega}$, and by the  uniform bound of $\{v_i\}$ in $L^\infty(\tilde{\omega})$,  $|\mathbb{V} v_i^{p-1}|\leq C|\mathbb{V}|$, where $C$ is independent of $i$. Since $|\mathbb{V}|\in \mathfrak{W}^p_{\loc}(\Gw) \subset L^1_{\loc}(\Gw)$, the  dominated convergence theorem implies that 
$$\int_{\omega} |\mathbb{V}|\,|v_i^{p-1}-v^{p-1}|\, |u|\dx \rightarrow 0\qquad \forall u\in C_c^\infty(\tilde{\omega}).$$
Moreover, $V_i\rightarrow \mathbb{V}$ in the quasinorm of $\mathfrak{W}^p_{\loc}(\Gw)$, and    $\mathfrak{W}^p_{\loc}(\Gw)$ is continuously embedded in $L^1_{\loc}(\Omega)$ (Corollary~\ref{cor_2}), it follows that  the first integral of the right-hand side of \eqref{eq18} converges to $0$ as $i\rightarrow \infty$. Thus,  \eqref{heqn_5} follows.
\end{proof}
\subsection{The Dirichlet problem}
The present subsection is devoted to the existence of a solution of the nonhomogeneous Dirichlet problem \eqref{dirch_prb}  in a bounded Lipschitz subdomain.  For a uniqueness result, see Theorem~\ref{existence_th}.  We begin with a lemma concerning the weakly lower semicontinuity and coerciveness of a functional $J$ associated with this Dirichlet problem.
\begin{lem}\label{lem_wlc_coer}
	Let  $1<p<d$, and $\omega\Subset \tilde{\omega}\Subset \omega'\Subset \mathbb{R}^d$, where $\gw$ is a Lipschitz domain.  Assume that $A$ and $V$ satisfy Assumptions~\ref{assump}, and let $\mathcal{G}\in \mathfrak{W}^p(\omega')$. Define a functional $J: W^{1,p}(\omega) \rightarrow \mathbb{R}\cup \{\infty\}$ by 
	$$J[u]:= \mathcal{Q}_{p,A,V}(u;\omega) -\int_\omega \mathcal{G} |u|\dx=\int_\omega |\nabla u|_{A}^p \dx + \int_\omega V|u|^p\dx -\int_\omega \mathcal{G} |u|\dx.$$
	Then:
	\begin{itemize}
		\item[(i)] The functional $J$ is weakly lower semicontinuous in $W^{1,p}(\omega)$.
		\item[(ii)] Moreover, if $V\geq 0$, then for any $g\in W^{1,p}(\omega)$ the functional $J$ is coercive in 
		$$S_g:=\{u\in W^{1,p}(\omega)\mid u=g \,\,\, \mbox{on } \partial \omega \mbox{ in the  trace sense} \}.$$
	\end{itemize} 
\end{lem}
\begin{proof}
(i). Let $\{u_n\}$ be a sequence in $W^{1,p}(\omega)$ such that $u_n\rightharpoonup u$ in $W^{1,p}(\omega)$. Consequently, $\{\|u_n\|_{W^{1,p}(\omega)}\}$ is bounded.  Since $W^{1,p}(\omega)$ is compactly embedded in $L^p(\omega)$, it follows that,  up to a subsequence,  $u_n\rightarrow u$ in $L^p(\omega)$ and  $u_n\rightarrow u$ a.e. in $\omega$.  \\
Let $\delta>0$. Using Minkowski's inequality and Theorem~\ref{cor_3}, we get
\begin{align}
&\left(\int_\omega V^\pm |u_n|^p \dx\right)^{1/p}-\left(\int_\omega V^\pm |u|^p \dx\right)^{1/p}
 \leq \left(\int_\omega V^\pm |u_n-u|^p\dx\right)^{1/p}\nonumber\\
& \leq \left(\delta \|\nabla(u_n-u)\|^p_{L^p(\omega, \mathbb{R}^d)} +C(d, p, \delta, \|V^\pm\|^{1/(p-1)}_{\mathfrak{W}^p(\omega')})\|u_n-u\|^p_{L^p(\omega)}\right)^{1/p}\nonumber\\
&\leq \delta^{1/p}\left(C+\|\nabla u\|_{L^p(\omega, \mathbb{R}^d)}\right) +C(d, p, \delta, \|V^\pm\|^{1/(p-1)}_{\mathfrak{W}^p(\omega')})\|u_n-u\|_{L^p(\omega)},\nonumber
\end{align}
where $C=\underset{n\in\mathbb{N}}{\sup} \|u_n\|_{W^{1,p}(\omega)}$. This implies that 
$\underset{n\rightarrow\infty}{\lim\sup}\int_\omega V^\pm |u_n|^p \dx \leq \int_\omega V^\pm |u|^p \dx$.
On the other hand, Fatou's lemma implies 
$\int_\omega V^\pm |u|^p \dx \leq \underset{n\rightarrow\infty}{\lim\inf}\int_\omega V^\pm |u_n|^p \dx$. Therefore, 
$$\underset{n\rightarrow\infty}{\lim}\int_\omega V |u_n|^p \dx = \int_\omega V|u|^p \dx.$$
In addition, the weak lower semicontinuity of the gradient term is immediate since the mapping $\xi\mapsto |\xi|^p_{A(x)}$ is convex. Thus,
$$\mathcal{Q}_{p,A,V}(u)\leq \underset{n\rightarrow \infty}{\liminf} \mathcal{Q}_{p, A, V}(u_n).$$
Similarly, for the last term of the functional $J$, we have
\begin{align*}
&\int_\omega \mathcal{G}^\pm |u_n| \dx- \int_\omega \mathcal{G}^\pm |u| \dx
 \leq \|\mathcal{G}^\pm\|^{1/p'}_{L^1(\omega)}\left(\int_\omega \mathcal{G}^\pm |u_n-u|^p\dx\right)^{1/p}\\
&\leq \|\mathcal{G}^\pm\|^{1/p'}_{L^1(\omega)}\!\!\left[\delta^{1/p}\left(\!C\!\!+\!\|\nabla u\|_{L^p(\omega, \mathbb{R}^d)}\!\right) \!\!+\!C(d, p, \delta, \|\mathcal{G}^\pm\|^{1/(p-1)}_{\mathfrak{W}^p(\omega')})\|u_n-u\|_{L^p(\omega)}\!\right]\!.
\end{align*}
This shows that $\underset{n\rightarrow\infty}{\lim\sup}\int_\omega \mathcal{G}^\pm |u_n| \dx \leq \int_\omega \mathcal{G}^\pm |u| \dx.$
On the other hand, by Fatou's lemma we have 
$\int_\omega \mathcal{G}^\pm |u| \dx \leq \underset{n\rightarrow\infty}{\lim\inf}\int_\omega \mathcal{G}^\pm |u_n| \dx.$ Thus,  
$$\underset{n\rightarrow\infty}{\lim}\int_\omega \mathcal{G} |u_n| \dx = \int_\omega \mathcal{G} |u| \dx,$$
and part (i) is proved.

(ii) Fix a number $t\in\R$, and suppose $u\in S_g$ such that $J[u]\leq t$. In order to prove that  the functional $J$ is coercive in $S_g$, it is enough to show that $\|u\|_{W^{1,p}(\omega)}\leq C$, where the positive constant $C$
is independent of $u$.\\
Since $V\geq 0$ in $\omega$ and $J[u]\leq t$, it follows that
\begin{align}
&\int_\omega |\nabla u|_A^p \dx\leq t +\int_\omega \mathcal{G} |u|\dx \nonumber\\
&\leq t+ \|\mathcal{G}\|^{1/p'}_{L^1(\omega)}\left( \int_\omega |\mathcal{G}||u|^p \dx\right)^{1/p}
 \leq t + C\|u\|_{W^{1,p}(\omega)},\label{eq:1}
\end{align}
where the last inequality follows from the Morrey-Adams theorem (Theorem~\ref{cor_3}), and the constant $C$ depends only on $d, p, \delta, \omega, \omega', \|\mathcal{G}\|^{1/(p-1)}_{\mathfrak{W}^p(\omega')}$ and $\|\mathcal{G}\|_{L^1(\omega)}$. Since $A$ is locally uniformly elliptic,  \eqref{eq:1} implies 
\begin{equation}\label{eq:2}
\|\nabla u\|^p_{L^p(\omega,\mathbb{R}^d)} \leq C_1 + C_2 \|u\|_{W^{1,p}(\omega)},
\end{equation}
where the positive constants $C_1$ and $C_2$ are independent of $u$. Since $u=g$ on $\partial \omega$, it follows  $u-g\in W_0^{1,p}(\omega)$. Furthermore,  by the Poincar\'e  inequality in $W_0^{1,p}(\omega)$, the uniform ellipticity of $A$ and inequality \eqref{eq:1}, we have 
\begin{align*}
\|u\|_{L^p(\omega)}&\leq \|u-g\|_{L^p(\omega)}+ \|g\|_{L^p(\omega)}
\leq C_p\|\nabla(u-g)\|_{L^p(\omega,\mathbb{R}^d)} + \|g\|_{L^p(\omega)}\\
& \leq C_p(\|\nabla u\|_{L^p(\omega,\mathbb{R}^d)}) + \|\nabla g\|_{L^p(\omega,\mathbb{R}^d)}) +\|g\|_{L^p(\omega)}\\
&\leq \frac{C_p}{\Theta_\omega} \left(\Big(\int_\omega |\nabla u|_A^p \dx\Big)^{1/p} +\|\nabla g\|_{L^p(\omega,\mathbb{R}^d)}\right) +\|g\|_{L^p(\omega)}\\
& \leq \frac{C_p}{\Theta_\omega} \left(\left( t+ C\|u\|_{W^{1,p}(\omega)}\right)^{1/p}+\|\nabla g\|_{L^p(\omega,\mathbb{R}^d)}\right) +\|g\|_{L^p(\omega)},
\end{align*}
where the constant $C_p$ depends only on $d, \omega$ and $C$ as in \eqref{eq:1}. Therefore, 
\begin{equation}
\|u\|_{L^p(\omega)}^p  \leq C_3 +C_4 \|u\|_{W^{1,p}(\omega)},\label{eq:3}
\end{equation}
where $C_3$, $C_4$ are positive numbers independent of $u$. Thus, from \eqref{eq:2} and \eqref{eq:3} we get 
$$\|u\|_{W^{1,p}(\omega)}^p  \leq C_5 +C_6 \|u\|_{W^{1,p}(\omega)},$$
where the positive constants $C_5$ and $C_6$ are independent of $u$. Since $p>1$ the latter inequality implies that there exists $C>0$ (independent of $u$) such that $\|u\|_{W^{1,p}(\omega)}  \leq C$.
\end{proof}
Consequently, we have: 
\begin{theorem}\label{dirchlet_prb}
	Assume $1<p<d$ and let $\omega\Subset \Gw$ be any Lipschitz subdomain. Let $A$ satisfy Assumptions~\ref{assump}, and let $V,\mathcal{G}\in \mathfrak{W}^p_{\loc}(\Gw)$, where $V, \mathcal{G}\geq 0$ a.e. in $\omega$,  and let $ g\in W^{1,p}(\omega)\cap C(\bar{\omega})$ be nonnegative. 
	Then there exists a nonnegative solution $v\in W^{1,p}(\omega)\cap C(\bar{\omega})$ to the Dirichlet problem 
	\begin{align}\label{dirch_prb}
	\begin{cases}
	Q_{p,A,V}(u) = \mathcal{G} & \mbox{ in }  \omega, \\
	u =g  &\mbox{ on }   \partial \omega.
	\end{cases}
	\end{align}
\end{theorem}
\begin{proof}
	Consider the functional $I: W^{1,p}(\omega) \rightarrow \mathbb{R}\cup \{\infty\}$ defined by 
	$$I[u]:= \mathcal{Q}_{p,A,V}(u;\omega) -\int_\omega \mathcal{G} u\dx=\int_\omega |\nabla u|_{A}^p \dx + \int_\omega V|u|^p\dx -\int_\omega \mathcal{G} u\dx.$$
	Let $\{u_n\}\subset S_g=\{u\in W^{1,p}(\omega)\mid u\!=\!g  \mbox{ on } \partial \omega\}$ be such that 
	$I[u_n]\searrow m:=\underset{u\in S_g}{\inf}I[u]$. Since $g\geq 0$, we  have $\{|u_n|\}\subset S_g$. Now, consider the functional $J: W^{1,p}(\omega) \rightarrow \mathbb{R}\cup \{\infty\}$ defined as in Lemma~\ref{lem_wlc_coer}. Since $\mathcal{G}\geq 0$, it follows that  $J[u]\leq I(u)$ and  $m\leq I[|u_n|]=J[u_n]\leq I[u_n]$. This implies that $J[u_n]\rightarrow m$ as $n\rightarrow \infty$ and $m=\underset{u\in S_g}{\inf}J[u]$.\\
	By Lemma~\ref{lem_wlc_coer}, the functional $J$ is weakly lower semicontinuous and also coercive in $S_g$. So,  in light of Mazur's theorem, and since $J$ is even,  there exists a nonnegative $u\in S_g$ such that $J[u]=m$. Also we have $J[u]=I[u]=m$. Thus, $u$ is a minimizer of the functional $I$ in $S_g$, and hence, a solution of problem \eqref{dirch_prb}.
\end{proof}
\subsection{Principal eigenvalue and the maximum principle}
In this subsection we recall the definition of a principal eigenvalue of the operator $Q_{p,A,V}$ satisfying Assumptions~\ref{assump}, and discuss its properties.
\begin{definition}[Eigenvalue]{\em 
	Let $\omega\Subset\Omega$ be a bounded domain. 
	
	1. A real number $\Lambda$ is said to be an {\em eigenvalue with an eigenfunction $v$} of the Dirichlet eigenvalue problem 
	\begin{align}\label{evp}
		\begin{cases}
			Q_{p,A,V}(u)=\Lambda |u|^{p-2}u & \mbox{ in } \omega,\\
			u=0 & \mbox{ on } \partial \omega,
		\end{cases}
	\end{align}
if $v\in W_0^{1,p}(\omega)\setminus \{0\}$ and $v$ satisfies
$$\int_\omega |\nabla v|_A^{p-2}A\nabla v\cdot \nabla\varphi\dx+ \int_\omega V|v|^{p-2}v\varphi \dx=\Lambda \int_\omega |v|^{p-2}v\varphi \dx \quad \forall \varphi \in C_c^\infty(\omega).$$
 
 2. A {\em principal eigenvalue}  of the Dirichlet eigenvalue problem \eqref{evp} is an eigenvalue of \eqref{evp} with a nonnegative eigenfunction.
}\end{definition}
The following theorem on the existence of principal eigenvalue extends the result in \cite[Thorem 3.9]{pinchover_psaradakis}, where the potential $V$ is assumed to be in the Morrey space $M^q(p;\omega)$. The proof follows the same arguments as in \cite[Thorem 3.9]{pinchover_psaradakis} (using the Morrey-Adams theorem for $V\in\mathfrak{W}^p(\omega)$), and therefore it is omitted.
\begin{theorem}
	Assume that $A$ and $V$ satisfy Assumptions~\ref{assump}. Then $Q_{p,A,V}$ on a subdomain  $\omega\Subset \Gw$ admits a unique principal eigenvalue, denoted by $\Lambda_1$. Moreover,  $\Lambda_1$ is given by the Rayleigh-Ritz variational formula 
	$$\Lambda_1:=\Lambda_1(Q_{p,A,V}; \omega)=\Lambda_1(\omega)=\inf_{u\in W_0^{1,p}(\omega)\setminus\{0\}}\frac{\mathcal{Q}_{p,A,V}(u;\omega)}{\|u\|^p_{L^p(\omega)}}.$$ 
	Furthermore, the principal eigenvalue $\Lambda_1$ is simple and isolated.
\end{theorem}
We recall a generalized  Picone identity which implies the strict monotonicity of $\Lambda_1(Q_{p,A,V}; \omega)$  as as a function of subdomains  $\gw \Subset \Omega$.  For the proof of Picone identity see \cite[Lemma 3.1]{PR} and Remark~3.2 therein.
\begin{lemma}[Picone identity]\label{pincone}
	Let $v>0, u\geq 0$ be in $W_{\loc}^{1,p}(\Omega)$ such that $uv^{-1}\in L^\infty_{\loc}(\Omega)$, and assume that $A$ and $V$ satisfy Assumptions~\ref{assump}. For $x\in \Omega$ define
	\begin{equation}\label{picone_eq1}
		L_A(u,v)(x):= |\nabla u|_A^p+(p-1)\frac{u^p}{v^p}|\nabla v|_A^p-p\frac{u^{p-1}}{v^{p-1}}\nabla u\cdot A(x)\nabla v|\nabla v|_A^{p-2}, 
	\end{equation}
and 
\begin{equation}\label{picone_eq2}
	R_A(u,v)(x):=|\nabla u|_A^p-\nabla\left(\frac{u^p}{v^{p-1}}\right)\cdot A(x)\nabla v|\nabla v|_A^{p-2}.
\end{equation}
Then $$L_A(u,v)(x)=R_A(u,v)(x)\geq 0 \quad \mbox{a.e. } \,x\in \Omega.$$
Moreover, $L_A(u,v)=0$ a.e.~in $\Omega$ if and only if $u=cv$ in $\Omega$ for some constant $c\geq 0$.
\end{lemma}
Using this Picone identity we have the following theorem on the positivity of principal eigenvalue which implies that if $\mathcal{Q}_{p,A,V}(\varphi)\geq 0 $ for all $\varphi\in C_c^\infty(\Omega)$, then $\Lambda_1(\omega)>0$ for any bounded subdomain $\omega\Subset\Omega$. This extends the result obtained in \cite[Lemma 5.1]{PR} for potential $V\in L^\infty_{\loc}(\Omega)$.
\begin{theorem}\label{cn_positivity}
	Let $\omega_1\Subset\omega_2\Subset\Omega$ be Lipschitz subdomains. Suppose that $A$ and $V$ satisfy Assumptions~$\ref{assump}$, and $\mathcal{Q}_{p,A,V}\geq 0$ in $\omega_2$. Then $\Lambda_1(\omega_1)>\Lambda_1(\omega_2)\geq 0$.
\end{theorem}
\begin{proof}
	From the Rayleigh-Ritz variational formula, we get $\Lambda_1(\omega_2)\geq 0$. Let $\phi_i\in W^{1,p}_{0}(\omega_i)$ be the normalized principal eigenfunctions corresponding to the principal eigenvalues $\Lambda_1(\omega_i)$, $i=1,2$. By the Harnack inequality, $\phi_i$ does not vanish in $\omega_i$ for $i=1,2$, and hence $\frac{\phi_1}{\phi_2}\in L^\infty_{\loc}(\omega_1)$.  Now consider a nonnegative minimizing sequence $\{\varphi_n\}\subset C_c^\infty(\omega_1)$ that converges to $\phi_1$ in $W^{1,p}_{0}(\omega_1)$. Then by using the Picone identity (Lemma~\ref{pincone}) we have
	\begin{align*}
		0\leq &\int_{\omega_2}L_A(\varphi_n, \phi_2)\d x =  \int_{\omega_1}L_A(\varphi_n, \phi_2)\d x=  \int_{\omega_1}R_A(\varphi_n, \phi_2)\d x\\
	&=	\int_{\omega_1}|\nabla \varphi_n|_A^p\d x- \int_{\omega_1}\nabla \left(\frac{\varphi_n^p}{\phi_2^{p-1}}\right)\cdot A(x)\nabla \phi_2|\nabla \phi_2|_A^{p-2}\d x\\
	&= \int_{\omega_1}|\nabla \varphi_n|_A^p\d x + \int_{\omega_1} V \varphi_n^p\d x -\Lambda_1(\omega_2)\int_{\omega_1}\varphi_n^p\d x, 
	\end{align*}
where we have used $\frac{\varphi_n^p}{\phi_2}\in W_0^{1,p}(\omega_2)$ as a test function in weak formulation of $\eqref{evp}$ in $\omega_2$. Letting $n\rightarrow \infty$ and using Fatou's lemma, we obtain
\begin{align*}
0\leq \int_{\omega_1}L_A(\phi_1, \phi_2)\d x &\leq \lim_{n\rightarrow\infty}\int_{\omega_1} L_A(\varphi_n, \phi_2)\d x\\
&= \left(\Lambda_1(\omega_1)-\Lambda_1(\omega_2)\right)\int_{\omega_1}\varphi_1^p\d x.	
\end{align*} 
This implies that $\Lambda_1(\omega_1)\geq \Lambda_1(\omega_2)$. If $\Lambda_1(\omega_1)= \Lambda_1(\omega_2)$, then $L_A(\phi_1, \phi_2)=0$ a.e. in $\omega_1$. Therefore,  Lemma~\ref{pincone}, implies $\phi_2|_{\omega_1}=c\phi_1$ for some $c>0$. But by the Harnack inequality  $\phi_2>0$ on $\partial\omega_1$ which contradicts the fact that $\phi_1=0$ on $\partial \omega_1$ in the trace sense. Thus, $\Lambda_1(\omega_1)>\Lambda_1(\omega_2)$. 
\end{proof}
The positivity of the principal eigenvalue in a bounded subdomain implies the generalized maximum principle for  $\mathcal{Q}_{p,A,V} \geq 0$.
\begin{theorem}[Maximum principle] \label{mp}
	Suppose that $A$ and $V$ satisfy Assumptions~\ref{assump}, and assume that $\mathcal{Q}\geq 0$ in $\Gw$. Let $\omega\Subset \Gw$ be a subdomain. Then the following assertions are equivalent:
\begin{itemize}
	\item[(i)] The principal eigenvalue $\Lambda_1$ of $Q_{p,A,V}$ in $\gw$ is positive.
	\item[(ii)] The weak maximum principle for $Q_{p,A,V}$ holds true in $\gw$: If $v\in W^{1,p}(\omega)$ is a solution of 
	\begin{align}\label{mp_e}
		\begin{split}
			Q_{p,A,V}(u)&=\mathcal{G} \quad \mbox{in } \omega, \\
			u&\geq 0 \quad \mbox{on } \partial\omega,
		\end{split}
	\end{align} 
where $\mathcal{G}\in L^1_{\loc} (\gw)$ is nonnegative, then $v\geq 0$ in $\omega$.
\item[(iii)] The strong maximum principle for $Q_{p,A,V}$ holds true in $\gw$: a solution $v\in W^{1,p}(\omega)$ of \eqref{mp_e} is either $v = 0$ or $v>0$ in $\omega$. 
\end{itemize}
\end{theorem}
\begin{proof}
	$\mbox{(i)}\Rightarrow \mbox{(ii)}$: Let $v\in W^{1,p}(\omega)$ be a solution of \eqref{mp_e}. Then choosing $v^{-}\in W_0^{1,p}(\omega)$ as a test function in the weak formulation of \eqref{mp_e} we have
	$$\mathcal{Q}_{p,A,V}(v^{-}; \omega)=\int_{\{x\in\omega: v(x)<0\}}\mathcal{G}v\dx.$$
 Since $\mathcal{G}\geq 0$ it follows that  $\mathcal{Q}_{p,A,V}(v^{-}; \omega)\leq 0$. If $v^-\neq 0$, then  $\Lambda_1\leq 0$, which contradicts (i).  Thus, $v^{-}=0$ a.e. in $\omega$ and,  hence $v\geq 0$ a.e. in $\omega$.
 
 $\mbox{(ii)}\Rightarrow \mbox{(iii)}$: Since $\mathcal{G}\geq 0$, it follows that  $v$ is a nonnegative supersolution of $Q_{p,A,V}(u)=0$ in $\omega$. The weak Harnack inequality (Theorem~\ref{w_hrnck}) implies that $v>0$ in $\omega$.
 
  $\mbox{(iii)}\Rightarrow \mbox{(i)}$: Suppose that the principal eigenvalue $\Lambda_1\leq 0$ and let $v\in W_0^{1,p}(\omega)$ be its principal eigenfunction. Then $u:=-v$ is a supersolution of $Q_{p,A,V}(u)=0$ in $\omega$ with $u=0$ on $\partial\omega$. Since $u\neq 0$, the strong maximum principle implies that $u>0$ which is a contradiction. Thus, $\Lambda_1>0$.
\end{proof}
Next we prove the existence and uniqueness for the Dirichlet problem. 
\begin{theorem}\label{thm_unique}
	Assume that $A$ and $V$ satisfy Assumptions~\ref{assump}, and let $\omega \Subset \omega' \Subset \Omega$ be subdomains. Suppose that $\Lambda_1$ of $Q_{p,A,V}$ in $\gw$ is positive. Then for $0\leq g\!\in\! L^{p'}(\omega)$ there exists a unique positive solution in $W_0^{1,p}(\omega)$ of 
	\begin{align}
		\begin{cases}\label{ee3}
			Q_{p,A,V}(u)= g & \mbox{ in } \omega, \\
			u= 0 & \mbox{ on } \partial\omega.
		\end{cases}
	\end{align} 
\end{theorem}
\begin{proof}
	Consider the functional 
	$$J[u]:=\mathcal{Q}_{p,A,V}(u;\omega)-\int_\omega gu\dx \qquad u\in W_0^{1,p}(\omega).$$
	Using Theorem~\ref{cor_1},  and following the same steps as in Lemma~\ref{lem_wlc_coer} (i), we obtain that the functional $\mathcal{Q}_{p,A,V}$ is weakly lower semicontinuous in $W_0^{1,p}(\omega)$.  On the other hand,  since $g\in L^{p'}(\omega)$, the functional $I(u):=\int_\omega gu \dx$ is a bounded linear functional in $W_{0}^{1,p}(\omega)$. Therefore, the functional $J[u]$ is weakly lower semicontinuous in $W_0^{1,p}(\omega)$.
	
	We claim that the functional $J[u]$ is coercive in $W_0^{1,p}(\omega)$. Assume that $J[u]\leq t$. Since $\Lambda_1>0$ hence we have $$\Lambda_1\|u\|^p_{L^p(\omega)}\leq\mathcal{Q}_{p,A, V}(u;\omega)\quad\forall u\in W_0^{1,p}(\omega).$$ 
	Then by applying H\"older inequality, we get 
	$$\Lambda_1 \|u\|^p_{L^p(\omega)}\leq J[u]+\int_\omega gu\dx \leq t+\|g\|_{L^{p'}(\omega)}\|u\|_{L^p(\omega)}.$$
	This inequality implies that 
	\begin{equation}\label{ee1}
		\|u\|_{L^p(\omega)}\leq m:=\max \left\{1, \left(\frac{t+ \|g\|_{L^{p'}(\omega)}}{\Lambda_1}\right)^{1/{(p-1)}}\right\}.
	\end{equation}
	On the other hand, applying H\"{o}lder's inequality and Morrey-Adams theorem  (Theorem~\ref{cor_1})  in the inequality $J[u]\leq t$ we obtain
	\begin{align}
		\int_\omega |\nabla u|_A^p\dx &\leq t +\int_\omega g u\dx + \int_\omega |V| |u|^p\dx\nonumber\\
		& \leq t + \|g\|_{L^{p'}(\omega)}\|u\|_{L^p(\omega)}+ \delta \|\nabla u\|^p_{L^p(\omega;\mathbb{R}^d)}+C_\delta \|u\|^p_{L^p(\omega)},\label{ee2}
	\end{align}
where $C_\delta=C(d,p,\delta, \omega,\omega', \|V\|_{\mathfrak{W}^p(\omega')})$.	Then, \eqref{ee1}, \eqref{ee2} and  Assumptions~\ref{assump} imply 
$$\|\nabla u\|^p_{L^p(\omega;\mathbb{R}^d)}\leq (\Theta_{\gw}^{-p}- \delta)^{-1}(t + \|g\|_{L^{p'}(\omega)}m+C_\delta m^p),$$
for $\gd< \Theta_{\gw}$. This together with \eqref{ee1} implies that  $\|u\|_{W^{1,p}(\omega)}\leq C$, where the constant $C>0$ is independent of $u$. 

Thus, the functional $J[u]$ is weakly lower semicontinuous and coercive in $W_0^{1,p}(\omega)$. Consequently, the functional has a minimizer in $W_0^{1,p}(\omega)$ and the corresponding Dirichlet problem has a nonnegative solution  $v_1\in W_0^{1,p}(\omega)$. Since $\Lambda_1>0$ by the strong maximum principle (Theorem~\ref{mp}), the solution $v_1$ is either $0$ or positive in $\omega$. 

{\bf Uniqueness:} Let $v_i \in W_0^{1,p}(\omega)$, $i=1,2$, be two solutions of \eqref{ee3}. If $v_1=0$, then $g=0$ in $\omega$, and then $v_2$ is a principal eigenvalue with eigenvalue $0$, but  in light of the uniqueness of the principal eigenvalue, this contradicts the assumption that $\Lambda_1>0$. 

Thus, $v_i>0$ in $\omega$ for $i=1,2$.  Consider $v_{i,h}:=v_i+h$, where $h$ is a positive constant and $i=1,2$. Taking $\varphi_{1,h}:=\frac{v_{1,h}^p-v_{2,h}^p}{v_{1,h}^{p-1}}\in W_0^{1,p}(\omega)$ as a test function in the definition of $v_1$ being a solution of the equation \eqref{ee3}, we have 
 \begin{align*}
	\int_{\omega}\!(v_{1,h}^p\! - \! v_{2,h}^p)|\nabla \log v_{1,h}|^p_{A}\!\dx\! & -\! p\!\int_\omega \!\!v_{2,h}^p|\nabla \log v_{1,h}|^{p-2}_A \! A\nabla \!\log v_{1,h} \! \cdot \! 
	\nabla \!\log(\frac{v_{2,h}}{v_{1,h}})
	\!\dx\\
	&= \int_\omega\frac{(g-Vv_{1}^{p-1})(v_{1,h}^p-v_{2,h}^p)}{v_{1,h}^{p-1}}.
\end{align*}
Similarly, using $\varphi_{2,h}= \frac{v_{2,h}^p-v_{1,h}^p}{v_{2,h}^{p-1}}\in W_0^{1,p}(\omega)$ as a test function in the definition of $v_2$ being a solution of the Dirichlet problem \eqref{ee3}, we get 
\begin{align*}
	\int_{\omega}\!(v_{2,h}^p\!-\!v_{1,h}^p)|\nabla \!\log v_{2,h}|^p_{A}\!\dx & \!- \! p\!\int_\omega \!\!v_{1,h}^p|\nabla \log v_{2,h}|^{p-2}_A\!A\nabla \!\log v_{2,h} \!\cdot\! \nabla \!\log(\frac{v_{1,h}}{v_{2,h}})\! \dx\\
	&= \int_\omega\frac{(g-Vv_{2}^{p-1})(v_{2,h}^p-v_{1,h}^p)}{v_{2,h}^{p-1}}.
\end{align*}
Adding the above two equations, we reach at
\begin{align*}
	& \int_\omega \!v_{1,h}^p\!\left(\! |\nabla \!\log\! v_{1,h}|^p_{A} \!-\! |\nabla \!\log\! v_{2,h}|^p_{A} \!- \! p|\nabla \log v_{2,h}|^{p-2}_A\!A\nabla\! \log v_{2,h} \!\cdot \! \nabla \! \log(\frac{v_{1,h}}{v_{2,h}})\!\right) \!\!\dx\\
	&+\!\!\int_\omega \!\! v_{2,h}^p\!\!\left(\!|\nabla \!\log\! v_{2,h}|^p_{A}\!-\! |\nabla \!\log\! v_{1,h}|^p_{A} \!-\!p\nabla \!\log v_{1,h}|^{p-2}_A\!A\nabla\! \log\! v_{1,h} \!\cdot\! \nabla\! \log(\frac{v_{2,h}}{v_{1,h}})\!\right)\!\!\dx\\
	&\,\,\,\,\,=\int_\omega \left[g\left(\frac{1}{v_{1,h}^{p-1}}-\frac{1}{v_{2,h}^{p-1}}\right)-V\left(\frac{v_1^{p-1}}{v_{1,h}^{p-1}}-\frac{v_2^{p-1}}{v_{2,h}^{p-1}}\right)\right](v_{1,h}^p-v_{2,h}^p)\dx
\end{align*}
Using the convexity inequality \eqref{eq:strong-conv}
we arrive at
\begin{align*}
	\!&\int_\omega \left[g\left(\frac{1}{v_{1,h}^{p-1}}-\frac{1}{v_{2,h}^{p-1}}\right)-V\left(\frac{v_1^{p-1}}{v_{1,h}^{p-1}}-\frac{v_2^{p-1}}{v_{2,h}^{p-1}}\right)\right](v_{1,h}^p-v_{2,h}^p)\dx \geq  C(p) \,\times \nonumber	\\[2mm]
	\!&  \int_\omega \!\!\!\dx \!\begin{cases}\!
		\!(v_{1,h}^p \!+ \! v_{2,h}^p)|\nabla \log(\frac{v_{1,h}}{v_{2,h}})|_A^p \qquad \qquad  p\in [2,\infty), \\[2mm]
		\! (v_{1,h}^p\! +\!v_{2,h}^p)|\nabla \!\log(\frac{v_{1,h}}{v_{2,h}})|_A^2\!\left(|\nabla\! \log\! v_{1,h}|_A \! + \! |\nabla\! \log\! v_{2,h}|_A\!\right)^{p-2}  \, \ p \! \in \! (1,2).
	\end{cases}
\end{align*}
This implies that 
$$0\!\geq \!\int_\omega \!\!g\left(\frac{1}{v_{1,h}^{p-1}}-\frac{1}{v_{2,h}^{p-1}}\right)(v_{1,h}^p-v_{2,h}^p)\!\dx\geq 
\!\!\int_\omega \! V \!\!\left(\frac{v_1^{p-1}}{v_{1,h}^{p-1}}-\frac{v_2^{p-1}}{v_{2,h}^{p-1}}\right)(v_{1,h}^p-v_{2,h}^p) \!\dx. $$
Since $v_{i,h}\rightarrow v_i$ as $h\rightarrow 0$ for $i=1,2$, it follows that the integrand of the  integral of the right hand side  above converges to $0$ a.e. in $\omega$. Also, 
$$\left|V \!\!\left(\frac{v_1^{p-1}}{v_{1,h}^{p-1}}-\frac{v_2^{p-1}}{v_{2,h}^{p-1}}\right)(v_{1,h}^p-v_{2,h}^p)\right |\leq 2 |V|(\left(v_1+1)^p+(v_2+1)^p\right)\in L^1(\omega).$$
Therefore, by the dominated converges theorem, it follows that 
$$\lim_{h\rightarrow 0}\int_\omega g\left(\frac{1}{v_{1,h}^{p-1}}-\frac{1}{v_{2,h}^{p-1}}\right)(v_{1,h}^p-v_{2,h}^p)\dx=0.$$
This together with the Fatou lemma implies that 
$$0=\int_\omega \!\!\!\dx \!\begin{cases}\!
\!(v_{1}^p \!+ \! v_{2}^p)|\nabla \log(\frac{v_{1}}{v_{2}})|_A^p &  p\in [2,\infty), \\[2mm]
\! (v_{1}^p\! +\!v_{2}^p)|\nabla \!\log(\frac{v_{1}}{v_{2}})|_A^2\!\left(|\nabla\! \log\! v_{1}|_A \! + \! |\nabla\! \log\! v_{2}|_A\!\right)^{p-2}  & \ p \! \in \! (1,2).
\end{cases}$$
Thus, $v_1=v_2$ a.e. in $\omega$.
\end{proof}
\subsection{Agmon-Allegretto-Piepenbrink-type theorem}
\hspace{-6pt} The next result, known as Agmon-Allegretto-Piepenbrink (AAP)-type theorem, asserts that the nonnegativity of the functional $\mathcal{Q}_{p,A,V}$ on $C_c^\infty(\Omega)$ is equivalent to the existence of a weak positive solution or supersolution of $Q_{p,A,V}(u)=0$ in $\Omega$. This result generalizes the AAP Theorem $4.3$ obtained in \cite{pinchover_psaradakis} for a potential $V\in M^q_{\loc}(p; \Omega)$ (see \cite{pinchover_psaradakis} for the history of AAP theorem).
\begin{theorem}[AAP-type theorem] \label{aap_thm} Suppose that  $A$ and $V$ satisfy Assumptions~\ref{assump}. Then the following assertions are equivalent:
	\begin{itemize}
		\item[(i)] $\mathcal{Q}_{p,A,V}\geq 0$ in $\Omega$;
		\item[(ii)] the equation $Q_{p,A,V}(u)=0$ in $\Omega$ admits a positive solution;
		\item[(iii)] the equation $Q_{p,A,V}(u)=0$ in $\Omega$ admits a positive supersolution.
	\end{itemize}	
\end{theorem}
\begin{proof}
	$\mbox{(i)}\Rightarrow\mbox{(ii)}$: Fix a point $x_0\in \Omega$ and consider a sequence $\{\omega_i\}_{i\in\mathbb{N}}$ of Lipschitz compact exhaustion $\cup_{i \in\mathbb{N}}\omega_i=\Omega$ such that $x_0\in \omega_1$, and consider the following Dirichlet problem for $i\geq 2$
	\begin{align}\label{aap_e}
		\begin{split}
			Q_{p,A,V+1/i}(u)&=f_i \quad \mbox{ in } \omega_i,\\
			u&=0\quad \mbox{ on } \partial\omega_i,
		\end{split}
	\end{align}
where $f_i\in C_c^\infty(\omega_i\setminus\bar{\omega}_{i-1})\setminus\{0\}$ and $f_i\geq 0$. By  (i), we have 
$$\Lambda_1=\inf_{u\in W_0^{1,p}(\omega_i)\setminus\{0\}}\frac{\mathcal{Q}_{p,A,V+1/i}(u;\omega_i)}{\|u\|^p_{L^p(\omega_i)}}\geq \frac{1}{i}\quad \mbox{for all } i\geq 1.$$
This together with Theorem~\ref{thm_unique} implies that the Dirichlet problem \eqref{aap_e} admits a positive solution $v_i\in W_0^{1,p}(\omega_i)$. Set  $\omega'_i=\omega_{i-1}$. As  $\supp(f_i)\subset \omega_i\setminus\bar{\omega'}_{i}$,  we get for $i\geq 2$
$$\int_{\omega'_i}|\nabla v_i|_A A\nabla v_i\cdot \nabla u\dx+ \int_{\omega'_i}\left(V+\frac{1}{i}\right)v_i^{p-1}u\dx=0 \quad \forall u\in W_0^{1,p}(\omega'_i).$$
By Theorem~\ref{mdlus_cont},  $v_i\in C(\omega'_i)$ for all $i\geq 2$. Normalize $f_i$ to get $v_i(x_0)=1$ for all $i\geq 2$. The Harnack convergence principle (Theorem~\ref{hrnck_principle}) with  $A_i=A$ and $V_i=V+{1}/{i}$, implies that, up to a subsequence, $v_i\to v$, where $v$ is a positive solution of $Q_{p,A,V}(u)=0$ in $\Gw$.

$\mbox{(ii)}\Rightarrow\mbox{(iii)}$ is trivial.

$\mbox{(iii)}\Rightarrow\mbox{(i)}$: Let $\tilde{v}\in W_{\loc}^{1,p}(\Omega)$ be a positive supersolution of  the equation $Q_{p,A,V}(u)=0$ in $\Omega$. The weak Harnack inequality (Theorem~\ref{w_hrnck}) implies that $1/\tilde{v}\in L^\infty_{\loc}(\Omega)$. For any $\varphi\in C_c^\infty(\Omega)$, choose $|\varphi|^p\tilde{v}^{1-p}\in W_0^{1,p}(\Omega)$ as a test function and denote $T:=-|\nabla\log \tilde{v}|_A^{p-2}\nabla \log\tilde{v}$, then
$$(p-1)\int_{\Omega}|T|_A^{p'}|\varphi|^p \dx\leq p\int_\Omega |T|_A|\varphi|^{p-1}|\nabla \varphi|_A\dx + \int_\Omega V |\varphi|^p \dx.$$
 Using the Young's inequality: $pab\leq (p-1)a^{p'}+b^p$ with $a=|T|_A|\varphi|^{p-1}$ and $b=|\nabla \varphi|_A$, we obtain  that $\mathcal{Q}_{p,A,V}(\varphi)\geq 0$ for all $\varphi\in C_c^\infty(\Omega)$.
\end{proof}

\subsection{Null-sequence and ground state}
In this subsection we recall the notion of a null-sequence and a ground state corresponding to the functional $\mathcal{Q}_{p,A,V}$. We show that if $\mathcal{Q}_{p,A,V}$ admits a ground state $\Phi$ in $\Gw$, then  $\Phi$ is the unique (up to a multiplicative constant) positive supersolution of ${Q}_{p,A,V}$ in $\Omega$. Moreover, in a bounded domain, the positive principal eigenfunction is in fact a ground sate if $A$ and $V$ satisfy assumptions~\ref{assump}, and \ref{assump_strong}.
\begin{definition}[Null-sequence]{\em 
	Assume that $\mathcal{Q}_{p,A,V}\geq 0$ in $\Omega$. A sequence $\{u_n\}\subset W_c^{1,p}(\Omega)$ is said to be a {\em null-sequence of $\mathcal{Q}_{p,A,V}$ in} $\Omega$ if 
	\begin{itemize}
	\item[(a)] $u_n\geq 0$ for all $n\in\mathbb{N}$;
	\item[(b)] there exists an open set $K\!\Subset\!\Omega$ such that $\|u_n\|_{L^p(K)}\!=\!1$ for all $n\!\in\!\mathbb{N}$;
	\item[(c)] $\displaystyle{\lim_{n\to\infty} \mathcal{Q}_{p,A,V}(u_n)=0}$.
	\end{itemize}
}\end{definition}
\begin{definition}[Ground state]{\em 
	A positive $\Phi\in W^{1,p}_{\loc}(\Omega)$ is  an {\em (Agmon) ground state of $\mathcal{Q}_{p,A,V}$ in $\Omega$} if $\Phi$ is an $L^p_{\loc}(\Omega)$ limit of a null-sequence.
}\end{definition}
\begin{lem}\label{lem_2}
Assume that $\mathcal{Q}_{p,A,V}  \!\geq\! 0$ in $\Omega$, where $A$ and $V$ satisfy Assumptions~\ref{assump}. If $p<2$ assume further that Assumption \ref{assump_strong} is satisfied. Suppose  that  $\mathcal{Q}_{p,A,V}$ admits  a null-sequence $\{u_n\}\subset W_c^{1,p}(\Omega)$  in $\Omega$. Let $v$ be a positive supersolution of $Q_{p,A,V}(u)\!=\!0$ in $\Omega$, and consider the sequence $\{w_n\!:=\! u_n/v\}$. Then $\{w_n\}$ is bounded in $W_{\loc}^{1,p}(\Omega)$, and $\nabla w_n\rightarrow 0$ in $L^p_{\loc}(\Omega; \mathbb{R}^d)$.
\end{lem}
\begin{proof}
	Fix $K\Subset\Omega$ such that $\|u_n\|_{L^p(K)}=1$ for all $n\in\mathbb{N}$, and let $K\Subset \omega\Subset \Omega$ be  a Lipschitz subdomain. By the Minkowski inequality we have
	$$\|w_n\|_{L^p(\omega)} \leq \|w_n-\langle w_n\rangle_K\|_{L^p(\omega)}+ \langle w_n\rangle_K(\mathcal{L}^d(\omega))^{1/p},$$
	where $\langle w_n\rangle_K=\frac{1}{|K|}\int_K w_n(x) \dx$. Applying the  Poincar\'e inequality to the first term, and the weak Harnack inequality (Theorem~\ref{w_hrnck}) to the second term of the right hand side of the above inequality, we have
	\begin{align}
		\|w_n\|_{L^p(\omega)} &\leq C(d,p, K, \omega) \|\nabla w_n\|_{L^p(\omega;\mathbb{R}^d)} + \frac{1}{\inf_K v}\langle u_n\rangle_K(\mathcal{L}^d(\omega))^{1/p}\nonumber\\
		& \leq C(d,p, K, \omega) \|\nabla w_n\|_{L^p(\omega;\mathbb{R}^d)} + \frac{1}{\inf_K v} \left(\frac{\mathcal{L}^d(\omega)}{\mathcal{L}^d(K)}\right)^{1/p},\label{lemeq_1}
	\end{align}
where in the  second term of the latter inequality, we used the H\"older inequality, and $\|u_n\|_{L^p(K)}=1$.
Let
\[I(v,w_k):=\left\{
\begin{array}{ll}
C(\bar{p})\displaystyle{\int_\Gw v^{\bar{p}}|\nabla w_k|_A^{\bar{p}}\dx} & 1\leq p <2, \\[3mm]
C(p)\displaystyle{\int_\Gw v^p|\nabla w_k|_A^p\dx} & p\geq 2,
\end{array}
\right.
\]
where $\bar{p} \geq p$, and $C(p)$ and $C(\bar{p})$ are  the constants in \eqref{eq_pbar} and \eqref{eq:strong-conv}, respectively. Using \eqref{eq_pbar} and \eqref{eq:strong-conv} to obtain
\bea
I(v,w_k)
&  \leq &  
\int_\Gw |\nabla u_k|_A^p\dx-\int_\Gw  w_k^p|\nabla v|_A^p\dx-\int_\Gw  v|\nabla v|_A^{p-2}A\nabla v \!\cdot\! \nabla (w_k^p)\dx \nonumber \\  
& = & 
\int_\Gw |\nabla u_k|_A^p\dx-\int_\Gw  |\nabla v|_A^{p-2}A\nabla v \!\cdot \! \nabla(w_k^pv)\dx.
\eea
Since $v$ is a positive supersolution, we get
\be\label{common:geq2}
I(v,w_k)
\leq
\int_\Gw |\nabla u_k|_A^p\dx
+
\int_\Gw  Vu_k^p\dx = Q_{A,p,V}[u_k].
\ee
Using the definition of $I$, the weak Harnack inequality, (and for $p\leq 2$ use also H\"older inequality), we obtain from \eqref{common:geq2} that
\be\label{nablawk0pgeq2}
c_\gw\!\int_\gw \!|\nabla w_k|^{p}\!\dx \!\leq\! C(p) \int_\Gw  v^{p}|\nabla w_k|^{p}_{A}\!\dx\!
\leq\!
\mathcal{Q}_{A,p,V}[u_k]
\rightarrow 0   \mbox{ as }k\rightarrow\infty,
\ee
where $c_\gw>0$ is a positive constant. By the weak compactness of $W^{1,p}(\gw)$, we infer that (up to a subsequence)
\be\label{LpwkGradestimate}
\nabla w_k\rightarrow0\qquad \mbox{in }L^p_{\rm loc}(\Gw;\R^d).
\ee
By \eqref{lemeq_1} and \eqref{nablawk0pgeq2}, we have that $w_k$ is bounded in $W_{\rm loc}^{1,p}(\gw)$.
\end{proof}
\begin{remark}\label{rem_p_less_}{\em 
In Lemma~\ref{lem_2}, for $p < 2$ we may replace Assumption~\ref{assump_strong}  with the assumption that the gradient of the positive supersolution $v$ is locally bounded (see the proof of Proposition 4.11 in \cite{pinchover_psaradakis}). }
\end{remark}
As a consequence of this lemma, we have the following uniqueness result:
\begin{theorem}\label{ground_st_thm}
	Assume that $\mathcal{Q}_{p,A,V}\!\geq\! 0$ in $\Omega$, where $A$ and $V$ satisfy Assumptions~\ref{assump},  and if $p<2$ assume further that Assumption~\ref{assump_strong} is satisfied. Then any null-sequence of the functional $\mathcal{Q}_{p,A,V}$ converges, both in $L^p_{\loc}(\Omega)$ and a.e. in $\Omega$ to a unique (up to a multiplicative constant) positive supersolution (which is in fact a solution) of the equation $Q_{p,A,V}(u)\!=\!0$ in $\Omega$.
\end{theorem}
\begin{proof}
	By the AAP theorem (Theorem~\ref{aap_thm}), the equation $Q_{p,A, V}(u)=0$ in $\Omega$ admits a positive solution $v\in W_{\loc}^{1,p}(\Omega)$ and a positive supersolution $\tilde{v}\in W_{\loc}^{1,p}(\Omega)$. Suppose that $\{u_n\}$ is a null-sequence of the functional $\mathcal{Q}_{p,A,V}$ in $\Omega$. Consider $w_n:=u_n/\tilde{v}$. Lemma~\ref{lem_2} implies that $\nabla w_n\rightarrow 0$ in $L^p_{\loc}(\Omega; \mathbb{R}^d)$. Then the Rellich-Kondrachov theorem implies that up to a subsequence (still denoted by $\{w_n\}$),  $w_n\rightarrow c$ in $W^{1,p}_{\loc}(\Gw)$ for a constant $c\geq 0$. Since $\tilde{v}$ is locally bounded and bounded away from zero,  $u_n\rightarrow c\tilde{v}$ a.e. in $\Omega$ and also in $L^p_{\loc}(\Omega)$ (up to a subsequence). In particular,  $c=\|\tilde{v}\|_{L^p(K)}^{-1}>0$. It follows that any null-sequence $\{u_n\}$ converges (up to a positive multiplicative constant) to the same positive supersolution $\tilde{v}$. Since the positive solution $v$ is also a positive supersolution,  we conclude that $\tilde{v}=Cv$ for some $C>0$, and thus $\tilde{v}$ is also the unique positive solution of the equation $Q_{p,A,V}(u)=0$ in $\Omega$.
\end{proof}
\begin{lemma}\label{ground_lem1}
	Suppose that $A$ and $V$ satisfy Assumptions~\ref{assump}. Let $\omega\Subset\Omega$ be a subdomain. Then, up to a multiplicative constant, the positive principal eigenfunction of $\mathcal{Q}_{p,A, V}$   in $\omega$ is a ground state  of $\mathcal{Q}_{p,A, V-\Lambda_1}$ in $\omega$.
\end{lemma}
\begin{proof}
	Let $v_1$ be the positive principal eigenfunction of $\mathcal{Q}_{p,A, V}$   in $\omega$ with principal eigenvalue $\Lambda_1$. Fix $K\Subset\omega$ such that  $C_K:=\|v_1\|_{L^p(K)}>0$. Then the sequence $\{C^{-1}_Kv_1\}$ is a null-sequence and, $C^{-1}_Kv_1$ is a ground state of the functional $\mathcal{Q}_{p,A, V-\Lambda_1}$ in $\omega$.
\end{proof}
As a corollary of Theorem~\ref{ground_st_thm}, we have: 
\begin{theorem}\label{evalue_supersoln}
	Let $\omega\Subset\Omega$ be a subdomain. Assume that $A$ and $V$ satisfy Assumptions~\ref{assump},  and if $p<2$ assume further that Assumption~\ref{assump_strong} is satisfied.  Then the following assertions are equivalent:
	\begin{itemize}
		\item[(i)] the principal eigenvalue $\Lambda_1>0$;
		\item[(ii)] the equation $Q_{p,A,V}(u)=0$ in $\omega$ admits a strict positive supersolution in $W^{1,p}_0(\omega)$;
		\item[(iii)] the equation $Q_{p,A,V}(u)=0$ in $\omega$ admits a strict positive supersolution in $W^{1,p}(\omega)$.
	\end{itemize} 
\end{theorem}
\begin{proof}
	$\text{ (i) }\Rightarrow\text{ (ii) }$: Assume that the principal eigenvalue $\Lambda_1>0$ with a principal eigenvalue $v\in W^{1,p}_0(\omega)$. Then by the strong maximum principle, $v$ is a strict positive supersolution of $Q_{p,A,V}(u)=0$ in $\omega$. \\
	$\text{ (ii) }\Rightarrow\text{ (iii) }$: This is immediate.\\	
$\text{ (iii) }\Rightarrow\text{ (i) }$: Suppose that the equation $Q_{p,A,V}(u)=0$ in $\omega$ admits a strict positive supersolution in $W^{1,p}(\omega)$.  By the AAP theorem (Theorem~~\ref{aap_thm}), $\mathcal{Q}_{p,A,V}(\varphi)\geq 0$ for all $\varphi\in C_c^\infty(\omega)$. This implies that the principal eigenvalue $\Lambda_1\geq 0$.  Now if $\Lambda_1=0$, then in light of Theorem~\ref{ground_st_thm} and Lemma~\ref{ground_lem1}, the associated principal eigenfunction is the unique (up to a multiplicative constant) positive supersolution of the equation $Q_{p,A, V}(u)=0$ in $\omega$. This contradicts our assumption that the equation $Q_{p,A, V}(u)=0$ in $\omega$ admits a strict positive supersolution in $W^{1,p}(\omega)$. Therefore, $\Lambda_1>0$.
\end{proof}
\subsection{Weak comparison principle}
The first result in this subsection is  a simple version of the weak comparison principle between sub and supersolutions of the operator $Q_{p,A, V}$ when the potential $V\in\mathfrak{W}_{\loc}^p(\Gw)$ is nonnegative. We omit the proof since it is quite standard and follows exactly as in \cite[Lemma 5.1]{pinchover_psaradakis}, where the potential is assumed to be in $M^q(p;\omega)$. 
\begin{lemma}\label{sub_sup_com}
	Let $1<p<d$, and $\omega\Subset \Gw$ be a Lipschitz subdomain. Suppose that $A$ and $V$satisfy Assumptions~\ref{assump}, $\mathcal{G}\in \mathfrak{W}^p_{\loc}(\Gw)$, and $V\geq 0$ in $\omega$. Let $v_1$, $v_2$ be respectively a subsolution and a supersolution of the problem 
	$$Q_{p,A, V}(u)=\mathcal{G} \quad \mbox{in } \omega.$$
If $v_1\leq v_2$ a.e. on $\partial \omega$ in the trace sense, then $v_1\leq v_2$ a.e. in $\omega$.
\end{lemma}
\begin{theorem}\label{existence_th}
	Let $1<p<d$ and $ \omega \Subset \Gw$ be a Lipschitz subdomain.  Suppose that $A$ and $V$satisfy Assumptions~\ref{assump} and $g \in  \mathfrak{W}^p_{\loc}(\Gw)$, where $g \geq 0$. Let $f, u_1, u_2\in W^{1,p}(\omega)\cap C(\bar{\omega})$, where $f\geq 0$ a.e. in $\omega$ and 
	\begin{align*}
	Q_{p,A, V}(u_1) &\leq g \leq Q_{p,A, V}(u_2) \quad \mbox{in } \omega,\\
	u_1 & \leq f \leq u_2 \quad \mbox{on } \partial \omega,\\
	0& \leq u_1 \leq u_2 \quad \mbox{in } \omega.
	\end{align*}
	Then there exists a  nonnegative solution $u\in W^{1,p}(\omega)\cap C(\bar{\omega})$ of the Dirichlet problem  
	\begin{align}\label{eq:6}
	\begin{cases}
	Q_{p, A, V}(u) = g \quad & \mbox{ in } \omega, \\
	u = f \quad &\mbox{ on } \partial\omega,
	\end{cases}
	\end{align}
	such that $u_1 \leq u \leq u_2$ a.e. in $\omega$. Moreover, if $f>0$ a.e. on $\partial \omega$, then the solution $u$ of the Dirichlet problem \eqref{eq:6} is unique.
\end{theorem}
\begin{proof} {\bf Existence:} 
	Define the set $$S:= \{ v \in W^{1,p}(\omega)\cap C(\bar{\omega}): 0\leq u_1 \leq v \leq u_2\,\,\mbox{in } \omega \}.$$ For $v\in S$, consider the function $\mathcal{G}(x, v):= g(x)+ 2 V^{-}(x) v^{p-1}(x)$, $x\in \omega$. Since $0\leq v \in C(\bar{\omega})$, it follows that  $\mathcal{G}\in \mathfrak{W}^p(\omega)$ and also $\mathcal{G} \geq 0$ a.e. in $\omega$. For $v\in S$ consider the following Dirichlet problem 
	\begin{align}\label{eq:7}
	\begin{cases}
		Q_{p, A, |V|}(u) = \mathcal{G}(x, v)  \quad&\mbox{ in } \omega, \\
	u = f  \qquad & \mbox{ on } \partial\omega.
	\end{cases}
	\end{align}
By Theorem~\ref{dirchlet_prb}, the Dirichlet problem \eqref{eq:7} admits a solution $u\!\in\! W^{1,p}(\omega)$. Moreover,  the solution is unique. Indeed, if $u$ and $\tilde{u}$ are solutions of the above problem, then we have 
\begin{align*}
	&Q_{p, A, |V|}(u) = \mathcal{G}(x, v) = Q_{p, A, |V|}(\tilde{u})  & \mbox{in } \omega, \\
&u =  f = \tilde{u}  & \mbox{on } \partial\omega.
\end{align*}
 By applying Lemma~\ref{sub_sup_com} with $u$ as a subsolution and $\tilde{u}$ as a supersolution, we get $u\leq \tilde{u}$ in $\omega$. By interchanging the roles of  $u$  and $\tilde{u}$, it follows that $\tilde{u}\leq u$ in $\omega$. Hence,  $u=\tilde{u}$ in $\omega$.
 
 Define a map $\psi: S \rightarrow W^{1,p}(\omega)$ by $\psi(v)= u$, where $u$ is a solution of the problem \eqref{eq:7}. The mapping $\psi$ is well-defined in the set $S$ and by using  Lemma~\ref{sub_sup_com}, it can be seen that the map $\psi$ is monotone. Let $v\in W^{1,p}(\omega)\cap C(\bar{\omega})$ be a subsolution of the Dirichlet problem \eqref{eq:6}. Then $$ Q_{p,A, |V|}(v)= Q_{p,A, V}(v) +\mathcal{G}(x,v)-g(x)\leq \mathcal{G}(x,v)$$ 
 in the weak sense, and hence $v$ is a subsolution of the problem \eqref{eq:7}. Also, by the definition of $\psi$, $\psi(v)$ is a solution of \eqref{eq:7}. By Lemma~\ref{sub_sup_com}, we have $v\leq \psi(v)$ a.e. in $\omega$. Moreover, 
 \begin{align*} 
 Q_{p,A,V}(\psi(v))&= Q_{p,A,|V|}(\psi(v)) -\mathcal{G}(x, \psi(v)) +g(x)\\
 &= 2V^{-}(|v|^{p-2}v-|\psi(v)|^{p-2}\psi(v))+g(x) \leq g(x) \quad \mbox{in } \, \omega,
 \end{align*}
 in the weak sense. This shows that if $v$ is a subsolution of  the Dirichlet \eqref{eq:6}, then $\psi(v)$ is  also a subsolution of \eqref{eq:6} with $v\leq \psi(v)$ a.e. in $\omega$. Similarly,  it can be seen that if $v\in W^{1,p}(\omega)\cap C(\bar{\omega})$ be a supersolution of  \eqref{eq:6}, then $\psi(v)$ also a supersolution of \eqref{eq:6} with $\psi(v)\leq v$ a.e. in $\omega$. \\
 Now define the following two sequences
 $$\underaccent{\bar}{v}_0:= u_1,\quad \underaccent{\bar}{v}_n:=\psi(\underaccent{\bar}{v}_{n-1}) \quad \mbox{ and } \quad \bar{v}_0:= u_2, \quad \bar{v}_n:= \psi(\bar{v}_{n-1}).$$
 From the above arguments,  the sequences $\underaccent{\bar}{v}_n\nearrow \underaccent{\bar}{v}$ and $\bar{v}_n\searrow\bar{v}$ for every $x\in \omega$. Also due to Brezis-Lieb Lemma\cite[Theorem 1.9]{leib_loss},
  both the sequences converges in $L^p(\omega)$. By using a similar argument as in Theorem~\ref{hrnck_principle}, it can be seen that $\underaccent{\bar}{v}$ and $\bar{v}$ are fixed point of $\psi$ and both are solutions of the Dirichlet problem \eqref{eq:6} with $u_1\leq \underaccent{\bar}{v}\leq \bar{v}\leq u_2$ in $\omega$.
 
 \medskip
 
 \noindent {\bf Uniqueness:} Let $v_1, v_2\in W^{1,p}(\omega)\cap C(\bar{\omega})$ be positive solutions of \eqref{eq:6} with $f>0$ a.e. on $\partial \omega$. Then using $\varphi_1= \frac{v_1^p-v_2^p}{v_1^{p-1}}\in W_0^{1,p}(\omega)$ as a test function in the definition of $v_1$ being a solution of \eqref{eq:6}, we have 
 \begin{align*}
 	\int_{\omega}(v_1^p-v_2^p)|\nabla \log v_1|^p_{A}\dx & - p\int_\omega v_2^p|\nabla \log v_1|^{p-2}_AA\nabla \log v_1 \cdot \nabla \log(\frac{v_2}{v_1})\dx\\
 	&= \int_\omega\frac{(g-Vv_1^{p-1})(v_1^p-v_2^p)}{v_1^{p-1}}.
 \end{align*}
Similarly, using $\varphi_2= \frac{v_2^p-v_1^p}{v_2^{p-1}}\in W_0^{1,p}(\omega)$ as a test function in the definition of $v_2$ being a solution of the Dirichlet \eqref{eq:6}, we get 
\begin{align*}
	\int_{\omega}(v_2^p-v_1^p)|\nabla \log v_2|^p_{A}\dx & - p\int_\omega v_1^p|\nabla \log v_2|^{p-2}_AA\nabla \log v_2 \cdot \nabla \log(\frac{v_1}{v_2})\dx\\
	&= \int_\omega\frac{(g-Vv_2^{p-1})(v_2^p-v_1^p)}{v_2^{p-1}} .
\end{align*}
Adding the above two equations, we reach at
\begin{align*}
	& \int_\omega v_1^p\left( |\nabla \log v_1|^p_{A}-|\nabla \log v_2|^p_{A}-p|\nabla \log v_2|^{p-2}_AA\nabla \log v_2 \cdot \nabla \log(\frac{v_1}{v_2})\right)\dx\\
	&+\int_\omega v_2^p\left(|\nabla \log v_2|^p_{A}-|\nabla \log v_1|^p_{A}-p\nabla \log v_1|^{p-2}_AA\nabla \log v_1 \cdot \nabla \log(\frac{v_2}{v_1})\right)\dx\\
	&\,\,\,\,\,=\int_\omega g\left(\frac{1}{v_1^{p-1}}-\frac{1}{v_2^{p-1}}\right)(v_1^p-v_2^p)\dx\leq 0
\end{align*}
Now,\eqref{eq:strong-conv} implies that
\begin{align*}
 0\!\geq 
  \!C(p)\! \!\! \int_\omega \!\!\!\dx \!
\begin{cases*}
 	(v_1^p+v_2^p)|\nabla \log(\frac{v_1}{v_2})|_A^p \qquad \mbox{ if } p\geq 2,\\[2mm]
 \!(v_1^p\!+\!v_2^p)|\nabla \log(\frac{v_1}{v_2})|_A^2\left(|\nabla \log v_1|_A \! + \! |\nabla \log v_2|_A\right)^{p-2}\mbox{ if } 1 \! <p \! <\!2.
 \end{cases*}
\end{align*}
Consequently, $v_2\nabla v_1=v_1\nabla v_2$ a.e. in $\omega$. Hence, there is a constant $c$ such that $v_1=cv_2$ a.e. in $\omega$. Since $v_1=v_2=f$ on $\partial \omega$, we infer that   $v_1=v_2$ in $\bar{\omega}$. Therefore, the Dirichlet problem \eqref{eq:6} admits a  unique solution.
 \end{proof}
\begin{theorem}[Weak Comparison Principle]\label{wcp}
Let $\omega\Subset\Omega$. Assume that $A$ and $V$ satisfy Assumptions~\ref{assump}, and    $\mathcal{Q}_{p,A,V}\geq 0$ in $\Omega$. Suppose that $v_1, v_2\in W^{1,p}(\omega)\cap C(\bar{\omega})$ satisfy the following inequalities 
\begin{equation*}
	\begin{cases*}
	Q_{p,A,V}(v_2)=g\,\,\text{in}\,\,\omega,\\
	v_2 >0\,\,\,\,\,\text{on}\,\,\partial\omega,
	\end{cases*}\quad \mbox{and}\quad 
	\begin{cases*}
	Q_{p,A,V}(v_1)\leq Q_{p,A,V}(v_2) \,\,\text{in}\,\,\omega,\\
	v_1\leq v_2\,\,\,\,\,\text{on}\,\,\partial\omega.
	\end{cases*}
	\end{equation*}
	where $g\in \mathfrak{W}^p(\omega)$ with $g\geq 0$ a.e. in $\omega$. Then $v_1 \leq v_2$ a.e. in $\omega$.
\end{theorem}
\begin{proof}
    By the assumption on $v_2$, we have that  $v_2$ is a supersolution of  equation \eqref{plaplace} in $\omega$ and $v_2>0$ on $\partial\omega$.  Hence, by the strong maximum principle we have $v_2>0$ in $\bar{\omega}$. Define $c:=\max\{1, \frac{\max_{\bar{\omega}} v_1}{\min_{\bar{\omega}}v_2}\}$. Then $v_1\leq cv_2$ in $\bar{\omega}$. Now consider the following boundary value problem 
    \begin{equation}\label{eq:8}
	\begin{cases*}
	Q_{p,A,V}(u)=g\,\,\text{in}\,\,\omega,\\
	u=v_2\,\,\,\,\,\text{on}\,\,\partial\omega.
	\end{cases*}
	\end{equation}
By our assumption and the choice of the constant $c$,  it follows that $cv_2$ is a supersolution of the problem \eqref{eq:8} with $v_1\leq v_2 \leq c v_2$ on $\partial\omega$. Also note that $v_1$ is a subsolution of the problem \eqref{eq:8}. By applying Theorem~\ref{existence_th} with $u_1=v_1$ and $u_2=cv_2$, we have a solution $v$ of \eqref{eq:8} such that $v_1\leq v \leq cv_2$ in $\omega$ with $v=v_2$ on $\partial\omega$. Again by the strong maximum principle, we have $v>0$ in $\bar{\omega}$. Now, the uniqueness of the boundary value problem \eqref{eq:8} (Theorem~\ref{existence_th}) implies that $v=v_2$. Therefore, $v_1\leq v_2$ a.e. in $\omega$.
\end{proof}
  
 \section{Fuchsian-type singularity and positive Liouville-type theorem}\label{FCH}
 The present section is devoted to the study of positive Liouville-type theorems, Picard-type principles, and removable singularity theorems for the equation 
 \begin{equation}\label{fuch_p_laplace}
 Q_{p,A,V}(u)=-\Delta_{p,A}(u)+V|u|^{p-2}u=0 \quad\mbox{in } \Omega,
 \end{equation}
where the potential $V\in \mathfrak{W}^p_{\loc}(\Gw)$ has an isolated Fuchsian-type singularity at $\zeta$ which belongs to the (ideal) boundary of $\Gw$.  Since we permit the domain $\Gw$ to be unbounded and the singular point $\zeta=\infty$,  it is worthy to consider the one-point compactification $\hat{\R}^d:=\R^d\cup\{\infty\}$ of $\R^d$. We denote by $\hat{\Gw}$ the closure of $\Gw$ in $\hat{\R}^d$.

 Throughout this section, we assume that $\mathcal{Q}_{p,A,V}\geq 0$ in $\Omega$. Note that under this assumption the weak comparison principle in bounded subdomains is valid when $A$ and $V$ satisfy Assumptions~$\ref{assump}$. 

 For $R>0$, we denote by $\mathcal{A}_R$ the annuls  given by
 $$\mathcal{A}_R:=\{x\in\mathbb{R}^d \mid  {R}/{2}\leq |x|< {3R}/{2} \}.$$  
 The dilated domain $\Omega/R$ of $\Gw$ is defined by
 $$\Omega/R:= \{x\in \mathbb{R}^d\,|\, x=R^{-1}y,\, \text{where}\,y\in\Omega\}.$$
 
 \medskip
 We assume that the singular point $\zeta\in \partial \hat{\Gw}$ is an isolated component of $\partial\hat{\Gw}$ and  $\zeta $ is either equal to $0$ or $\infty$. We are interested in the behavior of positive solutions of \eqref{fuch_p_laplace} near $\gz$. In fact, we may assume that  $\Gw$ is one of the followings domains: $B_r\setminus\{0\}$, $\R^d\setminus B_r$ for some $r>0$,  $\R^d$ and $\R^d\setminus\{0\}$.

 With some abuse of notation, we write 
\vspace{-3mm} 
$$x\to\gz \;\;\mbox{if} \;\;  
\begin{cases} 
x \to 0 & \mbox{in } \R \, \mbox{ and } \gz =0,\\
x \to \infty & \mbox{in  }  \R \, \mbox{ and }  \gz = \infty. 
\end{cases} $$

\begin{defin}[Fuchsian-type singularity]\label{fuchsian}{\em 
		Let $1<p<d$, and $A$ and $V$ satisfy Assumptions~\ref{assump}. Let $\zeta\in \{0, \infty\}$ be an isolated point of $\partial\hat{\Omega}$. Then we say that  the operator $Q_{p,A,V}$ has a {\em Fuchsian-type singularity at} $\zeta$ if there exists a punctured neighbourhood  $\Omega'\subset \Omega$ of  $\gz$ such that
			\begin{itemize}
				\item[(i)] The matrix $A$ is bounded and uniformly elliptic in $\Omega'$, i.e., there exists a constant $C>0$ such that 
				\be\label{ell}
				C^{-1}|\xi|^2\leq |\xi|_{A(x)}^2 \leq C|\xi|^2
				\qquad \forall x\in\Gw'\mbox{ and }\xi \in\R^d.
				\ee
				\item[(ii)] There exists a positive constant $C$ and $R_0>0$ such that 
				\begin{equation}\label{fuchsian_eqn}
				\| V\|_{\mathfrak{W}^p(\mathcal{A}_R\cap\Omega')}\leq C  
				 \end{equation}
				for all $0<R<R_0$ if $\gz=0$, and $R>R_0$ if $\gz=\infty$.
				
			\end{itemize}
 A set $\mathcal{A}\subset \Gw$ is said to be an {\it essential set} with respect to the singular point $\zeta$ if there exist $a\in(0,1)$, $b\in(1,\infty)$, and a sequence $\{R_n\}$ of positive numbers converging to $\zeta$ such that 
$$\mathcal{A}=\bigcup_{n = 1}^\infty \mathcal{A}_n,\quad \text{where}\,\,\mathcal{A}_n=\{x\in \Gw: aR_n<|x|<bR_n\}.$$
		}
\end{defin}
Similar to the linear case as in \cite{pinchover} or for the quasi-linear cases as in \cite{frass_pinchover, giri_pinchover}, it is enough to assume that  inequalities \eqref{ell} and \eqref{fuchsian_eqn} hold only on an essential set $\mathcal{A}$ with respect to the singular point $\zeta$, i.e., there exists a constant $C>0$ (independent of $n$) such that
\begin{equation*}\label{fuchsian_ess}
\C^{-1}|\xi|^2\leq |\xi|_{A(x)}^2 \leq C|\xi|^2
\quad \forall x\in \mathcal{A} \mbox{ and }\xi\in\R^d, \mbox{ and }\quad 
\| V\|_{\mathfrak{W}^p(\mathcal{A}_n)}\leq C,
\end{equation*}
for some essential set $\mathcal{A}=\cup_{n = 1}^\infty \mathcal{A}_n$ of $\zeta$.
\begin{example}{\em 
	Let $1<p<d$, and $\Gw=\R^d\setminus\{0\}$. Consider the  $(p,A)$-Laplacian with the Hardy potential $V(x)=\lambda |x|^{-p}$, and the equation
	\begin{equation}\label{hardy_eqn}
		-\Delta_{p,A}(u)-\frac{\lambda}{|x|^p}|v|^{p-2}v=0 \quad \mbox{in } \Gw,
	\end{equation}
where the matrix $A$ satisfies  \eqref{ell} outside a compact set in $\Gw$. We assert that equation \eqref{hardy_eqn} has Fuchsian-type singularity both at $\zeta =0$ and $\zeta =\infty$.\\
 Let us check that \eqref{fuchsian_eqn} is satisfied near $\zeta=0$. It is enough to show that for sufficiently small $R>0$, \eqref{fuchsian_eqn} is fulfilled over the annular set $\mathcal{A}_R=\{x\in\Omega: R/2\leq |x|< 3R/2\}$.  Indeed, 
\begin{align*}
\|V\|^{\frac{1}{p-1}}_{\mathfrak{W}^p(\mathcal{A}_R)} &=\sup_{x\in\mathcal{A}_R}\int_0^{\diam(\mathcal{A}_R)}\frac{1}{r^{\frac{d-1}{p-1}}}\left[\int_{\mathcal{A}_R\cap B_r(x)}|y|^{-p}\d y\right]^{\frac{1}{p-1}} \d r\nonumber\\[2mm]
&\leq C\sup_{x\in\mathcal{A}_R}\frac{1}{R^{\frac{p}{p-1}}}\int_0^{3R} \frac{1}{r^{\frac{d-1}{p-1}}} r^{\frac{d}{p-1}}\d r \leq  C \frac{1}{R^{\frac{p}{p-1}}} R^{\frac{p}{p-1}} =C <\infty.
\end{align*}
Note that the constant $C>0$ is independent of $R$. Hence, \eqref{hardy_eqn} has a Fuchsian-type singularity at $\zeta=0$. The case $\zeta=\infty$ follows similarly for large enough $R>0$.
}
\end{example}
Next, we present a dilation process using the quasi-invariance of \eqref{fuch_p_laplace} under the scaling $x\mapsto Rx$, where  $R>0$. For $R>0$, let $A_R$ and $V_R$ be the scaled matrix and the scaled potential defined by 
$$A_R(x):= A(Rx),\quad V_R(x):=R^pV(Rx) \qquad x\in \Omega/R.$$
If $u$ is a solution of \eqref{plaplace}, then for any $R>0$, the function $u_R(x):=u(Rx)$ is a solution of the equation 
\begin{equation*}\label{dil_eqn}
-\Delta_{p,A_R}(u_R)+V_R(x)|u_R|^{p-2}u_R=0\quad \mbox{in } \Gw/R.
\end{equation*}
Now take an annular subset $\mathcal{A}_R:= (B_{3R/2}\setminus\bar{B}_{R/2})\cap\Gw$. Since $\zeta$ is an
isolated singular point of $\partial \hat{\Gw}$, it follows that for $R$ `near' $\zeta$, the set $\mathcal{A}_R/R$ is a fixed annular set $\tilde{\mathcal{A}}$. Moreover,  for such $R$ we have
\begin{align}
& \|V_R\|_{\mathfrak{W}^p(\mathcal{A}_R/R)}^{\frac{1}{p-1}} =\sup_{x\in \mathcal{A}_R/R}\int_0^{\diam(\mathcal{A}_R/R)}\frac{1}{r^{\frac{d-1}{p-1}}}\left[\int_{\mathcal{A}_R/R\cap B_r(x)} \!\!\! R^p|V(Ry)|\d y\right]^{\frac{1}{p-1}}\!\!\!\!\!\d r\nonumber\\
&=\sup_{x\in \mathcal{A}_R/R}\int_0^{\diam(\mathcal{A}_R/R)}\frac{R^{\frac{p-d}{p-1}}}{r^{\frac{d-1}{p-1}}}\left[\int_{\mathcal{A}_R\cap B_{Rr}(Rx)}|V(y)|\d y\right]^{\frac{1}{p-1}}\!\!\!\!\d r\nonumber\\
&=\sup_{Rx\in \mathcal{A}_R}\int_0^{\diam(\mathcal{A}_R)}\frac{R^{\frac{p-d}{p-1}}\cdot R^{\frac{d-1}{p-1}}}{(Rr)^{\frac{d-1}{p-1}}}\left[\int_{\mathcal{A}_R\cap B_{Rr}(Rx)}|V(y)|\d y\right]^{\frac{1}{p-1}}\!\!\!\!\d r\nonumber\\
&= \!\! \sup_{Rx\in \mathcal{A}_R}\int_0^{\diam(\mathcal{A}_R)}\!\! \!\!\! \frac{1}{(Rr)^{\frac{d-1}{p-1}}} \! 
\left[ \! \int_{\mathcal{A}_R\cap B_{Rr}(Rx)} \!\!\!|\!V(y)|\d y \! \right]^{\frac{1}{p-1}}\!\!\!\!\d (Rr) 
\! = \!\|V\|_{\mathfrak{W}^p(\mathcal{A}_R)}^{\frac{1}{p-1}}\nonumber.
\end{align}
Thus, if  $Q_{p,A,V}$ has a Fuchsian-type singularity at $\zeta$, then 
\begin{equation}
\|V_R\|_{\mathfrak{W}^p(\tilde{\mathcal{A}})}=\|V_R\|_{\mathfrak{W}^p(\mathcal{A}_R/R)} =\|V\|_{\mathfrak{W}^p(\mathcal{A}_R)}\leq C,\label{fucheqn_1}
\end{equation}
where $C>0$ is a positive constant independent of $R$. 

The limiting dilation process is defined as follows. Let $\{R_n\}$ be a sequence of positive numbers satisfying $R_n\rightarrow \zeta$, and set $Y:=\lim_{n\to\infty}\Gw/{R_n}$. Clearly,  $Y=(\mathbb{R}^d)^*= \mathbb{R}^d\setminus\{0\}$. Let $\zeta=0$ or $\zeta=\infty$, and assume that  
\begin{align}
&A_{R_n} \longrightarrow \mathbb{A}\quad \mbox{in the weak topology of } L^\infty_{\loc}(Y, \mathbb{R}^{d\times d}), \label{dilated_conv1} \mbox{and }\\
&  \begin{cases}V_{R_n}\longrightarrow \mathbb{V} \quad \mbox{in the weak toplogy of } \mathfrak{W}^p_{\loc}(Y) \mbox{ if } 1<p\leq 2,\\
V_{R_n}\longrightarrow \mathbb{V} \quad \mbox{with resect to the quasinorm of } \mathfrak{W}^p_{\loc}(Y) \mbox{ if } p>2,
\end{cases}\label{dilated_conv2}
\end{align}
as $n\rightarrow \infty$. Define the {\em limiting dilated equation} with respect to  \eqref{fuch_p_laplace} and the sequence $\{R_n\}$ (which satisfies the conditions \eqref{dilated_conv1} and \eqref{dilated_conv2}) by
\begin{align}\label{dilated_eqn}
&\mathcal{D}^{\{R_n\}}(Q)(w) =\mathcal{D}^{\{R_n\}}(Q_{p,A,V})(w):=Q_{p,\mathbb{A},\mathbb{V}}(w)\nonumber \\
&= -\Delta_{p,\mathbb{A}}(w)+ \mathbb{V}|w|^{p-2}w=0\quad \mbox{on } Y.
\end{align}
Next, we show that the property of having a Fuchsian-type singularity at $\gz$ is preserved  under  the limiting dilation process $Q_{p,A, V}\mapsto \mathcal{D}^{\{R_n\}}(Q)$.
\begin{lem}
Let $1<p<d$ and assume that $A$ and $V$ satisfy Assumptions~\ref{assump}. Suppose that \eqref{fuch_p_laplace} has a Fuchsian-type singularity at $\zeta\in\{0,\infty\}\subset \partial\hat{\Gw}$. Then  for any sequence $R_n \rightarrow \zeta$ which satisfies conditions \eqref{dilated_conv1} and \eqref{dilated_conv2}, the corresponding limiting dilated equation $\mathcal{D}^{\{R_n\}}(Q)(w)=0$ in $Y$ has Fuchsian-type singularity at $\zeta$.
\end{lem}
\begin{proof}
By our assumption, there exists an essential set  $\mathcal{A}=\cup_{n=1}^\infty \mathcal{A}_n$, where $\mathcal{A}_n=\{x\in \Gw: aR_n<|x|<bR_n\} $  and a constant $C>0$  such that $$\|V\|_{\mathfrak{W}^p(\mathcal{A}_n)} \leq C, \quad \forall \,n \in \mathbb{N}.$$
We claim that 
$$\|\mathbb{V}\|_{\mathfrak{W}^p(\mathcal{A}_n)}\leq C, \quad \forall \, n\in \mathbb{N}.$$
Since for each $n$, $\mathcal{A}_n/R_n$ is a fixed annular set $\tilde{\mathcal{A}}=\{ x\in\Gw: a<|x|<b\}$, we have for any fixed $k$ sufficiently large:  
\begin{align*}
\!\|\mathbb{V}\|_{\mathfrak{W}^p(\mathcal{A}/R_k)} =\|\mathbb{V}\|_{\mathfrak{W}^p(\tilde{\mathcal{A}})}  \! \leq \underset{n\rightarrow \infty}{\lim\inf}\|V_{R_n}\|_{\mathfrak{W}^p(\mathcal{A}_n/R_n)} \!\!= \!\underset{n\rightarrow \infty}{\lim\inf} \|V\|_{\mathfrak{W}^p(\mathcal{A}_n)}
\!\!\leq \!\!C.
\end{align*}
where $C>0$ is independent on $k$, and the weak convergence of $\{V_{R_n}\}$ is used for $1<p\leq 2$, while for $p>2$,  a Fatou type lemma on quasinormed space follows from \cite[Lemma 3.3 and Lemma 3.5]{caetano}. Similarly, $\C^{-1}|\xi|^2\leq |\xi|_{\mathbb{A}(x)}^2 \leq C|\xi|^2$ for all $x\in \mathcal{A}/R_k$, $k\geq k_0$. Therefore,
the lemma follows. 
\end{proof}
Let us now recall the notion of a regular point (see \cite{frass_pinchover, frass_pinchover2, giri_pinchover}) of the operator $Q_{p,A,V}$ which turns out to be closely related to positive Liouville-type theorems for  \eqref{fuch_p_laplace}. 
\begin{definition} {\em 
Let $\mathcal{G}_\zeta$ be the germs of all positive solutions $u$ of the equation $Q_{p,A,V}(w)=0$ in relative
punctured neighbourhoods of $\zeta\in \{0,\infty\}$. The singular point $\zeta$ is said to be a {\em regular point of the equation $Q_{p,A,V}(w)=0$}  if for any two positive solutions $u, v \in \mathcal{G}_\zeta$ 
    $$\underset{x \in \Omega'}{\lim_{x\rightarrow \zeta}}\frac{u(x)}{v(x)}\,\,\, \mbox{ exists in the generalized sense},$$
    where $\Omega'$ is a relative punctured neighbourhood of $\zeta$.}
\end{definition}
The following result states that for any $1 < p <\infty $ and $d \geq 2$,  the singular point $\zeta$ is a regular point of the equation $-\Delta_{p,\mathbb{A}}(w)=0$ in $\mathbb{R}^d\setminus \{0\}$, where $\mathbb{A}\in \R^{d\times d}$ is a fixed symmetric and positive definite matrix and $\gz$ is either $0$ or $\infty$. For a proof see \cite[Theorem 4.1]{giri_pinchover}.
\begin{theo}\label{regularpoint_lap}
Assume that $1<p<\infty$ and $d\geq 2$. Let $\mathbb{A}\in \R^{d\times d}$ be a fixed symmetric and positive definite matrix. Then
\begin{itemize}
	\item[(i)] for  $p\leq d$, $\zeta =0$ is a regular point of  $-\Delta_{p,\mathbb{A}} (w)=0$ in $\mathbb{R}^d\setminus{\{0\}}$.
	\item[(ii)]  for $p\geq d$, $\zeta=\infty$ is a regular point of $-\Delta_{p,\mathbb{A}} (w)=0$ in $\mathbb{R}^d$.
\end{itemize}       
\end{theo}
\noindent Next we recall  the notion of positive solutions of minimal growth \cite{frass_pinchover,PR, PT, giri_pinchover}.
\begin{defin}{\em 
		(1) Let $K\Subset \Omega$. A positive solution $u$ of \eqref{plaplace} in $\Omega\setminus K$ is said to be a {\it positive solution of minimal growth in a neighbourhood of infinity in} $\Omega$ if for any $K'\Subset\Omega$ with smooth boundary such that  $K\Subset \mbox{int}\, K'$, and any positive supersolution $v\in C(\Omega\setminus\mbox{int}\,K')$ of \eqref{plaplace} in $\Omega\setminus K'$, we have
		$$u\leq v\quad\mbox{ on } \partial K' \quad \Rightarrow \quad u\leq v \quad \mbox{ in }\Omega\setminus K'.$$
		
		(2) A positive solution of \eqref{plaplace} in $\Omega$ which has minimal growth in a neighbourhood of infinity in $\Omega$ is called a {\em minimal positive solution}  of \eqref{plaplace} in $\Gw$.
		
		(3) Let $\zeta\in \partial\hat{\Omega}$ and let $u$ be a positive solution of \eqref{plaplace} in $\Omega$. Then $u$ is said to be a {\it positive solution of minimal growth in a neighbourhood of} $\partial\hat{\Omega}\setminus\{\zeta\}$ if for any relative punctured neighbourhood $K\Subset \hat{\Omega}$ of $\zeta$ such that $\Gamma:= \partial K\cap \Omega$ is smooth, and any positive supersolution $v\in C((\Omega\setminus K)\cup \Gamma)$ of \eqref{plaplace} in $\Omega\setminus K$, we have 
		$$u\leq v\quad \mbox{ on } \Gamma \quad \Rightarrow \quad u\leq v\quad \mbox{ in } \Omega\setminus K.$$
	}
\end{defin}
Next we discuss the existence of a positive solution of minimal growth of $Q(u)=0$  in a neighbourhood of $\partial\hat{\Omega}\setminus\{\zeta\}$, where $\zeta\in \{0,\infty\}$ is an  isolated singular point. 
\begin{theorem}\label{mn_growth1}
	Let $1<p<d$. Suppose that $\mathcal{Q}_{p,A,V}\geq $ in $\Gw$, where $A, V$ satisfy Assumptions~\ref{assump}. Assume that the operator $Q_{p,A,V}$ has an isolated Fuchsian-type singularity at $\zeta\in\partial\hat{\Omega}$, where $\gz\in\{0,\infty\}$. Then the equation $Q_{p,A,V}(u)=0$ in $\Omega$ admits a positive solution  of minimal growth in a neighbourhood of $\partial\hat{\Omega}\setminus\{\zeta\}$.
\end{theorem}
\begin{proof}
	The proof follows the same steps as in  \cite[Theorem 5.7]{pinchover_psaradakis} and \cite[Proposition 3.9]{giri_pinchover}, in which $V$ is assumed to be in $M^q_{\loc}(p;\Gw)$. 
\end{proof}
\begin{remark}{\em 
		It turns out that $\Phi$ is a ground state of \eqref{plaplace} in $\Gw$ if and only if $\Phi$ is a positive minimal positive solution (see, \cite[Theorem~5.9]{pinchover_psaradakis}). }
\end{remark}
The next lemma concerns a monotonicity property for the quotient of two positive solutions near the isolated singular point $\zeta\in \{0,\infty\}$.  
\begin{lem}\label{lemma_1}
	Let $1<p<d$, and  $A$, $V$ satisfy Assumptions~\ref{assump}. Assume that $u, v\in \mathcal{G}_\zeta$ are defined in a punctured neighbourhood $\Omega'$ of $\zeta\in \{0,\infty\}$. For $r>0$,  denote $\partial E_A(r):=\{x\in\mathbb{R}^d: |x|_{A^{-1}}=r\}$  where $A\in\mathbb{R}^{d\times d}$ is a symmetric, positive definite matrix (in particular if $A=I$ then $\partial E_A(r)=S_r$). Define
	\be\label{eq_mM}
	m_r:=\underset{\partial E_A(r)\cap \Omega'}{\inf}\,\frac{u(x)}{v(x)}\,,\hspace*{0.5cm}M_r:=\underset{\partial E_A(r)\cap \Omega'}{\sup}\,\frac{u(x)}{v(x)}\,.
	\ee
	The following assertions hold: 
	\begin{itemize}
		\item[(i)] The functions $m_r$ and $M_r$ are finally monotone as $r\rightarrow\zeta$. In particular, there are numbers $0\leq m\leq M\leq \infty$ such that 
		\be\label{eq_lim_mM}
		m:=\underset{r\rightarrow\zeta}{\lim}\, m_r, \quad \mbox{and} \quad M:=\underset{r\rightarrow\zeta}{\lim}\, M_r.
		\ee
		\item[(ii)] Suppose further that $u$ and $v$ are both positive solutions of \eqref{fuch_p_laplace} in $\Omega$ of minimal growth in $\partial \hat{\Gw}\setminus\{\zeta\}$, then $0<m\leq M<\infty$ and $m_r\searrow m$ and $M_r\nearrow M$ when $r\rightarrow\zeta$.
	\end{itemize}
\end{lem}
The proof of Lemma~\ref{lemma_1}  is similar to  \cite[Lemma~4.2]{frass_pinchover} for the case when $A=I$ and $V\in L^\infty_{\loc}(\Omega)$. We also refer to \cite[Lemma 3.14]{giri_pinchover} in which $V\in M^q_{\loc}(p,\Gw)$. Part (ii) of Lemma~\ref{lemma_1} implies the following result.
\begin{cor}\label{cor11}
	Suppose that operator $Q_{p,A, V}$ has a Fuchsian-type isolated singularity at $\zeta \in\{0,\infty\}$. Let $u, v$ be two positive solutions of \eqref{fuch_p_laplace} of minimal growth in a neighbourhood of $\partial\hat{\Omega}\setminus\{\zeta\}$. Then
	$$mv(x)\leq u(x)\leq M v(x)\quad x\in\Omega,$$
	where $0<m\leq M<\infty$ are defined in \eqref{eq_lim_mM}.
\end{cor}
As in \cite{frass_pinchover, giri_pinchover}, where $V\in L_{\loc}^\infty(\Gw)$ or $V\in \mathfrak{W}^p_{\loc}(\Gw)$, the regularity at $\gz\in\{0,\infty\}$ imply a Picard-type principle and a positive Liouville-type theorem, respectively.
\begin{prop}\label{mn_growth}
	Suppose that the operator $Q_{p,A,V}$ has and an isolated  Fuchsian-type singularity at $\zeta\in\{0,\infty\}$ which is regular. Then  \eqref{fuch_p_laplace} admits a unique (up to a multiplicative constant) positive solution in $\Omega$ of minimal growth in a neighbourhood of $\partial\hat{\Omega}\setminus\{\zeta\}$.
\end{prop}
\begin{proof}
	{\bf Existence:}  Follows from Theorem~\ref{mn_growth1}.
	
	{\bf Uniqueness:}	Let $u$ and $v$ be two solutions of  \eqref{fuch_p_laplace} of minimal growth in a neighbourhood of $\partial\hat{\Omega}\setminus\{\zeta\}$. Then Corollary~\ref{cor11} implies
	$$ mv(x)\leq u(x)\leq M v(x)\quad x\in\Omega,$$
	where $0<m\leq M<\infty$ are defined in \eqref{eq_mM} and \eqref{eq_lim_mM}.
	In addition, since $\zeta$ is a regular point, it follows that  
	$$\underset{x\in\Omega}{\displaystyle{\lim_{x\rightarrow \zeta}}} \, \frac{u(x)}{v(x)}\quad\mbox{ exists and is positive}.$$
	Therefore, we get $m=M$ and hence $u(x)=Mv(x)$.
\end{proof}
Let $\mathcal{A}_r:= (B_{3r/2}\setminus\bar{B}_{r/2})\cap \Omega'$ be an annular set, where $\Gw'$ is a punctured neighbourhood of $\gz$.  For $u, v\in \mathcal{G}_\zeta$, denote 
$$\mathbf{a}_r:= \underset{x\in \mathcal{A}_r}{\inf} \frac{u(x)}{v(x)}\,,\hspace*{1cm}\mathbf{A}_r:=\underset{x\in \mathcal{A}_r}{\sup} \frac{u(x)}{v(x)}\,.$$
The local Harnack inequality (Theorem~\ref{thm_hrnck}) implies that there exists $r_0>0$ such that 
$0< \mathbf{A}_r \leq C_r \mathbf{a}_r<\infty$ for every $0<r<r_0$. For a Fuchsian-type singularity of  \eqref{fuch_p_laplace} at $\zeta\in\{0,\infty\}$ we have the following uniform Harnack inequality in annular sets $\mathcal{A}_r$ for $r$ near $\zeta$. 
\begin{theo}[Uniform Harnack inequality]\label{u_hrnck}
	Let $1<p<d$.  Suppose that $A$  and $V$ satisfy Assumptions~\ref{assump}. Further assume that $Q_{p,A,V}$  has an isolated Fuchsian-type singularity at $\zeta  \in \{0,\infty\}$.  Let $u, v\in \mathcal{G}_\zeta$. Then there exists positive constant $C$ independent of $r$,  $u$ and $v$ such that 
	$$ \mathbf{A}_r\leq C\mathbf{a}_r\qquad \forall \,  r \mbox{ near } \zeta.$$
\end{theo}
\begin{proof}
	Let $\Gw'\subset \Gw$ be a fixed punctured neighbourhood of $\gz$ and $u, v\in \mathcal{G}_\zeta$ be two positive solutions in $\Gw'$. For $r>0$, consider the annular set $\tilde{\mathcal{A}}_r:= (B_{2r}\setminus\bar{B}_{r/4})\cap \Omega'$. By our assumption, $\zeta=0$ (or $\zeta=\infty$) is an isolated singular point. Therefore, for $r<r_0$ (respectively, $r>r_0$) $\mathcal{A}:=\mathcal{A}_r/r$ and $\tilde{\mathcal{A}}:=\tilde{\mathcal{A}}_r/r$ are fixed annulus with  $\mathcal{A}\Subset\tilde{\mathcal{A}}$.
	
	For such $r$, Let $u_r(x):=u(rx)$, $v_r(x):=v(rx)$, where $x\in\Omega'/r$. Then the functions $u_r$ and $v_r$ are positive solution of the equation 
	\begin{equation*}
		Q_r[u_r]:=-\Delta_{p,A_r}(u_r)+ V_r(x)|u_r|^{p-2}u_r=0 \,\,\text{in}\,\,\tilde{\mathcal{A}},
	\end{equation*}
	where $A_r(x):= A(rx)$ and $V_r=r^{p}V(rx)$.  By estimate \eqref{fucheqn_1},  the norms $\|V_r\|_{\mathfrak{W}^p(\tilde{\mathcal{A}})}$ of the scaled potentials are uniformly bounded $\tilde{\mathcal{A}}$.  Also, by \eqref{ell},  the scaled matrices $A_r(x)$ are uniformly elliptic  and bounded  in $\tilde{\mathcal{A}}$. Thus, by applying the local Harnack inequality (Theorem~\ref{thm_hrnck}) in the annular domain $\tilde{\mathcal{A}}$ we get
	$$\mathbf{A}_r=\underset{x\in \mathcal{A}_r}{\sup} \frac{u(x)}{v(x)}=\underset{x\in \mathcal{A}}{\sup} \frac{u_r(x)}{v_r(x)}\leq C\underset{x\in \mathcal{A}}{\inf} \frac{u_r(x)}{v_r(x)}=C\underset{x\in \mathcal{A}_r}{\inf} \frac{u(x)}{v(x)}=C \mathbf{a}_r,$$
	where the constant $C>0$ is independent of $u$ and $v$ for $r$ near $\gz$. 
\end{proof}
The next result is an extension to the case $V\in \mathfrak{W}^p_{\loc}(\Gw)$ of Proposition~3.18 in \cite{giri_pinchover}, where the potential $V$ is assumed to be in $M^q_{\loc}(p;\Omega)$.
\begin{theo}\label{regularity_thm}
	Let $1<p<d$, and  $A, V,$ satisfy Assumptions~\ref{assump}. Assume  that the operator $Q_{p,A,V}$ has an isolated Fuchsian-type singularity at $\zeta\in\partial\hat{\Omega}$, and there is a sequence $R_n\rightarrow \zeta$, such that either $0$ or $\infty$ is a regular point of a limiting dilated equation $\mathcal{D}^{\{R_n\}}(Q)(w)=0$ in $Y$. Then $\zeta$ is a regular point of the equation $Q_{p,A,V}(u)=0$ in $\Omega$.
\end{theo}
\begin{proof} The proof follows the same steps as in \cite[Proposition~3.18.]{giri_pinchover}. For the completeness, we provide the proof.
Let $u, v\in\mathcal{G}_\zeta$. Define
	\begin{equation*}\label{en1}
		m_r:= \underset{S_r\cap\Omega'}{\inf} \frac{u(x)}{v(x)},\quad M_r:= \underset{S_r\cap\Omega'}{\sup} \frac{u(x)}{v(x)},
	\end{equation*}
	where $\Omega'$ is a punctured neighbourhood of $\zeta$. By Lemma~\ref{lemma_1}, $M \! := \! \lim_{r\rightarrow\zeta}{M_r}$ and  $m:=\lim_{r\rightarrow\zeta}{m_r}$ exist in the generalized sense. Thus, we only need to  show that $M=m$. 
	
	If $M:=\lim_{r\rightarrow\zeta}{M_r}=\infty$ (respectively, $m:=\lim_{r\rightarrow\zeta}{m_r}=0$), then by the uniform Harnack inequality, Lemma~\ref{u_hrnck}, we get $m=\infty$ (respectively, $M=0$), and hence the limit 
	$$\underset{\substack{x\rightarrow\zeta\\x\in\Omega'}}{\lim}\frac{u(x)}{v(x)}\quad \text{exists in the generalized sense.}$$
	So, it can be assumed that $u\asymp v$ in some neighbourhood $\Omega'\subset \Omega$ of $\zeta$. Let us now fix $x_0\in \mathbb{R}^d$ such that $R_nx_0\in \Omega$ for all $n\in \mathbb{N}$ and consider
	\begin{equation*}\label{en2}
		u_n(x):= \frac{u(R_nx)}{u(R_nx_0)}\, ,\qquad v_n(x):=\frac{v(R_nx)}{u(R_nx_0)}\, .
	\end{equation*}
	Then by the definition of $\mathcal{G}_\zeta$, $u_n$ and $v_n$ are positive solutions of the equation
	\begin{equation*}\label{eqn}
		-\Delta_{p,A_{n}}(w)+ V_n(x)|w|^{p-2}w=0\quad \mbox{in } \Omega'/{R_n},
	\end{equation*}
	where $A_{n}(x):=A_{R_n}(x)$ and $V_n(x):=V_{R_n}(x)$, are the associated scaled matrix and potential, respectively. Since $u_n(x_0)=1$ and $v_n(x_0)\asymp 1$, the Harnack convergence principle (Proposition~\ref{hrnck_principle}) implies that the sequence $\{R_n\}$ admits a subsequence (still denoted by $\{R_n\}$) such that  
	\begin{equation*}
		\lim_{n\rightarrow\infty} u_n(x):=u_\infty(x), \mbox{ and } \lim_{n\rightarrow\infty}v_n(x):=v_\infty(x)
	\end{equation*}
	locally uniformly in $Y =\displaystyle{\lim_{n\rightarrow\infty} \Omega'/R_{n}}$, and $u_\infty$ and $v_\infty$ are positive solutions of the equation 
	\begin{equation}\label{dilated_eq}
		\mathcal{D}^{\{R_n\}}(Q)(w) = -\Delta_{p,\mathbb{A}}(w)+\mathbb{V}|w|^{p-2}w=0\qquad \text{in } Y.
	\end{equation}
	So,  for any fixed $R>0$ we get
	\begin{align*}
		\underset{x\in S_R}{\sup} \frac{u_\infty(x)}{v_\infty(x)} & \!=\!\underset{x\in S_R}{\sup}\lim_{n\rightarrow\infty}\frac{u_n(x)}{v_n(x)} =
		\lim_{n\rightarrow\infty}\underset{x\in S_R}{\sup}\frac{u_n(x)}{v_n(x)}\\[2mm]
		&\!=\!\lim_{n\rightarrow\infty}\underset{x\in S_R}{\sup}\frac{u(R_nx)}{v(R_nx)}\!=\!\lim_{n\rightarrow\infty}\underset{R_nx\in S_{RR_n}}{\sup}\frac{u(R_nx)}{v(R_nx)} \!=\! \lim_{n\rightarrow\infty} M_{RR_n} \!=\!M, 
	\end{align*}
	where we have used the existence of $\lim_{r\rightarrow\zeta}{M_r}=M$, and the local uniform convergence of the sequence $\{u_n/v_n\}$ in $Y$. Similarly, we get
	$$\underset{x\in S_R}{\inf} \frac{u_\infty(x)}{v_\infty(x)}=m.$$
	By our assumption, either $\zeta_1=0$ or $\zeta_1 =\infty$ is a regular point of the dilated equation \eqref{dilated_eq}, so the limit 
	$$\underset{\substack{x\rightarrow\zeta_1\\x\in Y}}{\lim}\frac{u_\infty(x)}{v_\infty(x)}\quad \text{exists.}$$
	Hence, $m=M$, this shows that 
	$$\underset{\substack{x\rightarrow\zeta\\x\in \Omega'}}{\lim}\frac{u(x)}{v(x)}\quad \text{exists.}$$
	Therefore,  $\zeta$ is a regular point of the equation $Q_{p,A,V}(u)=0$ in $\Omega$. 
\end{proof}
We now introduce the notion of a weak Fuchsian-type singularity, and prove a Liouville-type theorem for $Q_{p,A,V}$ having  a such singularity at $\gz$.
\begin{defin}{\em 
		Let $1<p<d$, and $A, V$ satisfy Assumptions~\ref{assump}.  Assume that the operator $Q_{p,A,V}$  has an isolated  Fuchsian-type singularity $\zeta\in\{0,\infty\}$. The operator $Q_{p,A,V}$ is said to have a {\em  weak Fuchsian-type singularity} at $\zeta$ if  there exist $m$ sequences $\{R_n^{(j)}\}_{n=1}^\infty \subset\mathbb{R}_+$, $1\leq j\leq m$, satisfying $R_n^{(j)}\rightarrow \zeta^{j}$, where $\zeta^{(1)}=\zeta$, and $\zeta^{(j)}=0$ or $\zeta^{(j)}=\infty$ for $2\leq j\leq m$, such that 
		\begin{equation*}\label{wfs}
			\mathcal{D}^{\{R_n^{(m)}\}}\circ\cdots\circ\mathcal{D}^{\{R_n^{(1)}\}}(Q)(u)=-\Delta_{p,\mathbb{A}} (u)\qquad \text{on }Y,
		\end{equation*}
		where  $\mathbb{A} \! \in \! \R^{d\times d}$  is a symmetric, positive definite matrix, and $Y\!= \!\!\underset{n\rightarrow \infty}{\lim} \Omega/R_n^{(1)}$.
	}
\end{defin}
\begin{theo}[Liouville theorem]\label{liouville_thm}
	Let $1<p<d$ and $A, V$ satisfy Assumptions~\ref{assump}.  Suppose that $\zeta\in\partial  \hat{\Omega}$ is an isolated singular point, where $\zeta =0$ or $\zeta=\infty$.  Further, assume that the operator $Q=Q_{p,A,V}$ has a weak Fuchsian-type singularity at $\zeta$. Then $\zeta$ is a regular point of \eqref{fuch_p_laplace}. 
	
	In other words, if $u$ and $v$ are two positive solutions of the equation $Q_{p,A,V}(w)=0$ in a punctured neighborhood of $\zeta$, then 
	\begin{itemize}
		\item[(i)] 
		$\;\;\displaystyle{\lim_{x\rightarrow\zeta}\frac{u(x)}{v(x)}}$  exists in the generalized sense.
		\item[(ii)] the equation $Q_{p,A,V}(w)=0$ admits a unique positive solution in $\Omega$ of minimal growth in a neighbourhood of $\partial\hat{\Omega}\setminus\{\zeta\}$.
	\end{itemize}
\end{theo}
\begin{proof}
	By Proposition~\ref{mn_growth}, we have (i) $\Rightarrow$ (ii). Thus, we only need to  show that  $\underset{x\rightarrow\zeta}{\lim}\,\frac{u(x)}{v(x)}$ exists in the generalized sense. Since the operator $Q_{p,A,V}$ has a weak Fuchsian-type singularity at $\zeta$, we have  
	\begin{equation}\label{lveqn1}
		\mathcal{D}^{\{R_n^{(m)}\}}\circ\cdots\circ\mathcal{D}^{\{R_n^{(1)}\}}(Q)(w)=-\Delta_{p,\mathbb{A}} (w)=0\qquad \mbox{in }\,\mathbb{R}^d\setminus\{0\},
	\end{equation}
	where $\mathbb{A}\in \R^{d\times d}$ is a symmetric, positive definite matrix. Recall that by Theorem~\ref{regularpoint_lap}  either $0$ or $\infty$ is a regular point of $-\Delta_{p,\mathbb{A}}$. Therefore, Theorem~\ref{regularity_thm} and a reverse induction argument imply that $\zeta$ is a regular point of the equation $Q_{p,A,V}(w)=0$.
\end{proof}
\begin{center}
	{\bf Acknowledgements}
\end{center}
{\small The  authors  acknowledge  the  support  of  the  Israel  Science Foundation (grant  637/19) founded by the Israel Academy of Sciences and Humanities.}

\end{document}